\pdfoutput=1
\documentclass[11pt]{article}

\usepackage[margin=1in]{geometry}

\usepackage{amsmath,amscd}
\usepackage{amssymb}
\usepackage{amsthm}
\usepackage{comment}
\usepackage{graphicx}
\usepackage{epstopdf} 
\usepackage{enumitem}
\usepackage{mathrsfs}
\usepackage{cite}
\usepackage[ocgcolorlinks, linkcolor=blue]{hyperref}
\usepackage{authblk}

\usepackage{bm}
\usepackage{hyperref}
\usepackage{xcolor}
\hypersetup{
	colorlinks,
	linkcolor={blue},
	urlcolor={blue},
	citecolor={red}
}

\newtheorem{thm}{Theorem}[section]

\newtheorem{lem}[thm]{Lemma}
\newtheorem{cor}[thm]{Corollary}

\newtheorem{defi}[thm]{Definition}
\newtheorem{remark}[thm]{Remark}

\newtheorem*{proposition*}{Proposition}

\numberwithin{equation}{section}

\newcommand{\norm}[1]{\lVert #1 \rVert}

\title{Determining habitat anomalies in cross-diffusion predator-prey chemotaxis models}

\author[1,*]{Yuhan Li}
\author[1,$\dagger$]{Hongyu Liu}
\author[2,$\natural$]{Catharine W. K. Lo}
\affil[1]{Department of Mathematics, City University of Hong Kong, Hong Kong, P. R. China}
\affil[2]{School of Mathematical Sciences, Shenzhen University, Shenzhen, 518061, P.R. China}
\affil[*]{yuhli2-c@my.cityu.edu.hk}
\affil[$\dagger$]{hongyliu@cityu.edu.hk, hongyu.liuip@gmail.com}
\affil[$\natural$]{catharinelowk@gmail.com, cwklo@szu.edu.cn}
\date{}

\begin{document}
\maketitle

\begin{abstract}
This paper addresses an open inverse problem at the interface of mathematical analysis and spatial ecology: the unique identification of unknown spatial anomalies -- interpreted as zones of habitat degradation -- and their associated ecological parameters in multi-species predator-prey systems with multiple chemical signals, using only boundary measurements. We formulate the problem as the simultaneous recovery of an unknown interior subdomain and the discontinuous ecological interaction rules across its boundary. A unified theoretical framework is developed that uniquely determines both the anomaly’s geometry and the discontinuous coefficients characterizing the altered interactions within the degraded region. Our results cover smooth anomalies in time-dependent systems and are extended to non-smooth polyhedral inclusions in stationary regimes. This work bridges a gap between ecological sensing and the quantitative inference of internal habitat heterogeneity, offering a mathematical basis for detecting and characterizing habitat degradation from limited external data.
~
\end{abstract}

{\centering \section{Introduction}  \label{sec:tp_intro} }
\subsection{Biology background and motivation of our study}\label{sec:tp_BioBack}
The intricate dynamics governing how cohabiting species interact with each other and their heterogeneous environments represent a cornerstone of spatial ecology. Real-world populations are distributed across landscapes fragmented by spatial anomalies such as zones of habitat degradation, where resource availability, toxicity, or physical structure differ sharply from the surrounding environment. To capture the resulting dynamics, models must account for both random dispersal and the directed movement of organisms in response to environmental cues across these discontinuous domains.

A fundamental mechanism for such directed movement is chemotaxis \cite{Grunbaum1998, cantrell2004spatial, Cantrell2020}, the ability to sense and move along chemical gradients. This process is central to diverse ecological functions and is acutely sensitive to habitat degradation, which can disrupt chemical signaling pathways. For instance, in soil ecosystems, chemotaxis drives essential microbial processes like nitrogen fixation and denitrification, where contamination or pH shifts can sever these chemical dialogues. Similarly, in the mammalian gut, chemotaxis underpins the complex spatial organization of the microbiota; inflammation or antibiotic exposure can degrade this chemical landscape, leading to dysbiosis. By integrating such taxis mechanisms, models bridge the scale from local individual interactions to the emergent, large-scale spatial patterns -- and their disruption -- observed in nature.

A powerful framework that integrates interspecies dynamics with spatial movement is the predator-prey chemotaxis system, which models interactions where movement is directed by chemical stimuli, typically a chemoattractant produced by the prey. An illustrative formulation, extending both the classic Lotka-Volterra predator-prey framework and the Keller-Segel chemotaxis paradigm \cite{Keller1970, Tello2012, Black2016, Stinner2014}, is given by:
\begin{equation}\label{eq:0ppc_intro}
    \begin{cases}
        \partial_t u = d_u \Delta u - \chi \nabla \cdot (u \nabla w) - a u v,\\
        \partial_t v = d_v \Delta v + b u v - c v,\\
        \partial_t w = d_w \Delta w + d u - e w,
    \end{cases}
\end{equation}
where \( u(x, t) \), \( v(x, t) \), and \( w(x, t) \) denote the densities of prey, predator, and chemoattractant, respectively. The parameters \( d_u, d_v, d_w \) are diffusion coefficients; \( \chi \) is the prey's chemotactic sensitivity; and \( a, b, c, d, e \) represent predation rate, predator growth, predator mortality, chemoattractant production, and its degradation. Critically, in a degraded zone, parameters such as the predation rate \( a \) or chemotactic sensitivity \( \chi \) may change discontinuously, modeling a sharp shift in ecological rules.

In this system, the dynamics has been further enriched by incorporating prey-taxis, which describes directed movement of prey in response to predator gradients or conspecifics \cite{Eisenbach2004,SB2009, RCP2011, Zhangglobal, Lin2024}. This mechanism, introduced by \cite{Kareiva1987} and developed by \cite{Grunbaum1998}, is crucial for modeling adaptive anti-predator strategies. For example, predatory bacteria like \emph{Bdellovibrio} exhibit prey-taxis towards chemical signals from \emph{Escherichia coli} prey. Such behavior can be modeled by adding a term \( \chi \nabla \cdot (u \nabla u) \) for self-aggregation or \( -\chi \nabla \cdot (u \nabla v) \) for predator-avoidance. These strategies are themselves vulnerable to disruption in degraded habitats.

Beyond taxis, cross-diffusion provides another essential mechanism for spatial pattern formation. Pioneered by Shigesada, Kawasaki, and Teramoto \cite{Shigesada1979}, it generalizes dispersal by allowing the population gradient of one species to directly influence the movement of another:
\begin{equation*}
    \begin{cases}
        \partial_t u = \nabla \cdot (d_u \nabla u + d_{1} \nabla v),\\
        \partial_t v = \nabla \cdot (d_v \nabla v + d_{2} \nabla u).
    \end{cases}
\end{equation*}
Here, coefficients \( d_1 \) and \( d_2 \) quantify interspecific dispersal pressures. Their sign is ecologically significant: negative values typically induce aggregation, while positive values promote segregation. In a degraded anomaly, these cross-diffusive couplings may flip sign or change magnitude, representing a breakdown or reversal of typical spatial interactions \cite{Luo2024, Dubey2001}. These dynamics are not merely theoretical; they are observed in complex microbial ecosystems, such as within biofilms, where interspecific chemical signaling can lead to attraction or repulsion, effectively manifesting as cross-diffusion and resulting in structured spatial organization.

In summary, these models form a sophisticated toolkit for forward modeling of population dynamics under spatial heterogeneity \cite{li2025unveiling}. The pervasive nature of habitat degradation and other spatial anomalies necessitates robust methods not only to model their effects but, crucially, to detect and characterize them non-invasively. However, the corresponding inverse problem -- deducing such anomalous subregions with discontinuous ecological rules from externally observed data -- remains a significant open challenge in spatial ecology.

This leads to the central question of this work: \emph{How can one uniquely identify and localize a subregion of habitat degradation within a domain, and simultaneously quantify the discontinuous ecological interaction rules inside it, using only population data measured at the domain's boundary over time?}

In this paper, we address this inverse problem for a class of spatially extended ecological systems featuring cross-diffusion and taxis. Our objective is the simultaneous recovery of both the geometric shape of an internal anomaly -- representing a zone of habitat degradation -- and the set of altered ecological parameters within it. Mathematically, this is formulated as the unique determination of an unknown interior subdomain and the piecewise-constant or discontinuous coefficients that govern species interaction and movement across its boundary. This framework constitutes a significant departure from prior inverse problem studies in mathematical biology, which have largely been confined to the recovery of smoothly varying or globally constant parameters \cite{LLL1, LLL2025inverse, LLL2025pee}.

Our results establish a mathematical foundation for non-invasive ecological sensing. A direct application is the localization and characterization of a subsurface contaminant plume by observing the spatiotemporal response of microbial communities at the domain's boundary. More broadly, the methodology is generalizable to a range of scenarios requiring the identification of hidden internal structures from boundary data, such as delineating strongholds of invasive species, detecting the onset of harmful algal blooms, or mapping regions of acute environmental stress within complex ecosystems.

\subsection{Mathematical setup}\label{sec:tp_Setup}
Motivated by the works above, we consider the following multi-species, multi-chemical cross-diffusion predator-prey chemotaxis system with prey-taxis. Let the space-time domain be $Q := \Omega \times (0, T]$ and its lateral boundary $\Sigma:= \partial\Omega \times (0, T]$. The governing system is given by:
\begin{equation}\label{eq:tp_maingeneral}
    \begin{cases}
    \partial_t u_1(x,t)-d_1\Delta \left(u_1(x,t)+\sum_{j=1}^M\delta_{1j} u_1(x,t)v_j(x,t)\right)\\
    \qquad \qquad \qquad\qquad \qquad \qquad\qquad \qquad \qquad =G_1(x,t,u_1,\dots,u_N,v_1,\dots,v_M)&\ \text{in } Q ,\\
    \qquad \qquad \qquad \qquad \qquad \qquad\qquad \qquad \qquad\vdots  &\quad \vdots\\   
    \partial_t u_N(x,t)-d_N\Delta \left(u_N(x,t)+\sum_{j=1}^M\delta_{Nj} u_N(x,t)v_j(x,t)\right)\\
    \qquad \qquad \qquad \qquad \qquad \qquad\qquad \qquad \qquad=G_N(x,t,u_1,\dots,u_N,v_1,\dots,v_M) &\ \text{in } Q ,\\
    \partial_t v_1(x,t)-\delta_1 \Delta v_1(x,t)=\chi^{1}_{1}\nabla \cdot (v_1 \nabla u_1)+ \cdots + \chi^{N}_{1}\nabla \cdot (v_1 \nabla u_N)&\ \text{in } Q ,\\
    \qquad \qquad \qquad \qquad \qquad \qquad \vdots \qquad \qquad \qquad \qquad \qquad \qquad \qquad &\quad \vdots\\   
     \partial_t v_M(x,t)-\delta_M \Delta v_M(x,t)=\chi^{1}_{M}\nabla \cdot (v_M \nabla u_1)+ \cdots + \chi^{N}_{M}\nabla \cdot (v_M \nabla u_N)&\ \text{in } Q ,\\
     %\partial_\nu u_1(x,t)=\cdots=\partial_\nu u_N(x,t)=0, \  \partial_\nu v_1(x,t)=\cdots=\partial_\nu v_M(x,t)=0 &\  \text{on } \Sigma,\\
     u_1(x,0)=f_1, \cdots, u_N(x,0)=f_N,\  v_1(x,0)=g_1, \cdots, v_M(x,0)=g_M,  &\  \text{in } \Omega.
    \end{cases}
\end{equation}

Here, $\mathbf{u} := (u_1, \dots, u_N)$ denotes the vector of chemical concentrations, and $\mathbf{v} := (v_1, \dots, v_M)$ denotes the vector of prey population densities. The parameters are defined as follows:  $d_i>0$ refers to diffusion coefficient of the $i$-th chemical, $\delta_j>0$ is intrinsic diffusion coefficient of the $j$-th prey species, and $\delta_{ij}>0$ represents cross-diffusion coefficient quantifying the influence of prey $v_j$ on the diffusion of chemical $u_i$. $\chi^i_j=\{0,1\}$ indicates whether the chemotactic movement of prey $v_j$  is directed by the gradient of chemical $u_i$. $G_i$ is reaction term governing the local production or consumption of the $i$-th chemical. The system couples Lotka–Volterra-type species–chemical interactions with Keller–Segel-type taxis mechanisms, capturing both interspecific dynamics and spatially oriented movement in response to chemical signals.

In the long time limit $T \to \infty$, the system approaches a steady state in which transient dynamics decay and temporal derivatives become negligible relative to the diffusive, advective, and reactive terms. Consequently, the densities of prey species and chemical concentrations converge to time-independent equilibrium profiles, reducing the system to its stationary form:
\begin{equation}\label{eq:tp_mainstationary}
    \begin{cases}    
    -d_1\Delta\left(u_1(x)+\sum_{j=1}^M\delta_{1j} u_1(x)v_j(x)\right)=F_1(x,u_1,\dots,u_N,v_1,\dots,v_M) &\quad \text{in } \Omega ,\\
    \qquad \qquad \qquad \qquad \qquad \qquad \qquad \qquad \vdots &\quad \vdots\\   
    -d_N\Delta \left(u_N(x)+\sum_{j=1}^M\delta_{Nj} u_N(x)v_j(x)\right)=F_N(x,u_1,\dots,u_N,v_1,\dots,v_M) &\quad \text{in } \Omega ,\\ 
    -\delta_1 \Delta v_1(x)=\chi^{1}_{1}\nabla \cdot (v_1 \nabla u_1)+ \cdots + \chi^{N}_{1}\nabla \cdot (v_1 \nabla u_N)&\quad \text{in } \Omega ,\\
    \qquad \qquad \qquad \vdots \qquad \qquad \qquad \qquad \qquad \qquad  &\quad \vdots\\   
     -\delta_M \Delta v_M(x)=\chi^{1}_{M}\nabla \cdot (v_M \nabla u_1)+ \cdots + \chi^{N}_{M}\nabla \cdot (v_M \nabla u_N)&\quad \text{in } \Omega ,\\
     u_1(x)=f_1, \cdots, u_N(x)=f_N, \ v_1(x)=g_1, \cdots, v_M(x)=g_M,  &\quad \text{on } \partial \Omega,
    \end{cases}
\end{equation}
where $F_i(x,u_1,\dots,u_N,v_1,\dots,v_M)$ is the stationary limit of $G_i(x,t,u_1,\dots,u_N,v_1,\dots,v_M)$ correspondingly for each $i\in \{1,\dots,N\}$.

Suppose there is an anomaly denoted by $\omega$, which is a subset of the habitat $\Omega$. This anomaly corresponds to a distinct region within $\Omega$, where the environmental conditions differ from the rest of the habitat. Specifically, this region might be either more hostile or more favorable for living. For example, it could have increased sunlight exposure or contain harmful chemicals and pesticides. Our objective is to identify and analyze this anomaly and its impact on the reaction functions $F$ and $G$ in equations \eqref{eq:tp_maingeneral} and \eqref{eq:tp_mainstationary}, respectively. 
 
We assume that $\omega$ is an open bounded $C^2$ domain within $\Omega$, with $\Omega \backslash \omega$ connected, and that the outer boundary $\partial\Omega$ is also of class $C^2$ for $\Omega \subset \mathbb{R}^n$, $n \in \mathbb{N}$. The reaction function $\mathbf{G}$ exhibits a discontinuity across $\partial\omega$. Specifically, $\mathbf{G}$ takes the form
 \begin{equation}\label{jump:G}
     \mathbf{G}(x,t,u_1,\dots,u_N,v_1,\dots,v_M)=\begin{cases}
         \mathbf{G}^1(x,t,u_1,\dots,u_N,v_1,\dots,v_M) \quad & \text{if} \ x \in \omega,\\
         \mathbf{G}^0(x,t,u_1,\dots,u_N,v_1,\dots,v_M) \quad & \text{otherwise},
     \end{cases}
 \end{equation}
with $\mathbf{G}^1,\mathbf{G}^0$ being $C^{\gamma}$ H\"older continuous for some $\gamma \in (0,1)$ in their respective domains, such that
\begin{equation}\label{eq:GSep}
    \mathbf{G}^1(x,t,u_1,\dots,u_N,v_1,\dots,v_M) \neq \mathbf{G}^0(x,t,u_1,\dots,u_N,v_1,\dots,v_M) \quad \text{for} \  x\in \partial\omega. 
\end{equation}

In the stationary case of equation \eqref{eq:tp_mainstationary}, we similarly assume the presence of a topological structure characterized by an anomalous inhomogeneity $\omega \Subset \Omega$. Here, $\omega$ is a bounded Lipschitz domain such that $\Omega \backslash \bar{\omega}$ is connected. We also assume that a similar form holds for the reaction function $\mathbf{F}(x,u_1,\dots,u_N,v_1,\dots,v_M)$, as in the time-dependent case for $G$ in \eqref{eq:GSep}. We will give more precise descriptions of these functions in subsection \ref{sec:tp_nonsm_ad}. 

Associated to these inverse problems, we construct the following measurement maps of a pair of Cauchy data:
\begin{equation}\label{eq:tp_MeasureMap1}
    \mathcal{M}^{+}_{\omega,\mathbf{G}}  := \left((\mathbf{u},\mathbf{v})|_{\partial\Omega}, (\mathbf{u},\mathbf{v})|_{t=T}\right) \to \omega,\mathbf{G},
\end{equation}
where $\nu$ is the exterior unit normal vector to $\partial\Omega$. 
Similarly, we have 
\begin{equation}\label{eq:tp_MeasureMap2}
    \mathcal{M}^{+}_{\omega,\mathbf{F}}  :=  (\partial_{\nu}\mathbf{u},\partial_{\nu}\mathbf{v})|_{\partial\Omega} \text{ fixed} \to \omega,\mathbf{F},
\end{equation}
where $\mathbf{f}=(f_1,\dots,f_N)$, $\mathbf{g}=(g_1,\dots,g_M)$ are the initial functions related to \eqref{eq:tp_mainstationary}.

The `+' sign in \eqref{eq:tp_MeasureMap1} and \eqref{eq:tp_MeasureMap2} indicates that the data are associated with the non-negative solutions of the models \eqref{eq:tp_maingeneral} and \eqref{eq:tp_mainstationary}, respectively. Physically speaking, the symbol $\omega$ in equation \eqref{eq:tp_MeasureMap1} denotes the unidentified smooth domain within $\Omega$, while the $\omega$ referenced in \eqref{eq:tp_MeasureMap2} signifies the region of the irregular polyhedral inclusion present in the space of $\Omega$. Therefore, the inverse problems \eqref{eq:tp_MeasureMap1} and \eqref{eq:tp_MeasureMap2} focus on recovering the location, shape, and parameter configuration of this anomaly. This is also known as the inverse inclusion problem in the theory of inverse problems.

Pertaining to the inverse inclusion problems \eqref{eq:tp_MeasureMap1} and \eqref{eq:tp_MeasureMap2}, we are interested in the following two unique identifiability issues:
\begin{itemize}
    \item Can we uniquely determine $\omega$ and $\mathbf{G}$ using the measurement map $\mathcal{M}^{+}_{\omega,\mathbf{G}}$? Specifically, if two admissible inclusions $(\omega_k,\mathbf{G}_k)$ for $k=1, 2$ produce the same boundary measurement, does it imply that $(\omega_1,\mathbf{G}_1)=(\omega_2,\mathbf{G}_2)$?  In other words, if $\mathcal{M}^+_{\omega_1,\mathbf{G}_1}=\mathcal{M}^+_{\omega_2,\mathbf{G}_2}$ in \eqref{eq:tp_maingeneral}, does it follow that $(\omega_1,\mathbf{G}_1)=(\omega_2,\mathbf{G}_2)$?
     \item Can we uniquely determine $\omega$ and $\mathbf{F}$ using the measurement map $\mathcal{M}^{+}_{\omega,\mathbf{F}}$? Specifically, if two admissible inclusions $(\omega_k,\mathbf{F}_k)$ for $k=1, 2$ produce the same boundary measurement, does it imply that $(\omega_1,\mathbf{F}_1)=(\omega_2,\mathbf{F}_2)$?  In other words, if $\mathcal{M}^+_{\omega_1,\mathbf{F}_1}=\mathcal{M}^+_{\omega_2,\mathbf{F}_2}$ in \eqref{eq:tp_mainstationary}, does it follow that $(\omega_1,\mathbf{F}_1)=(\omega_2,\mathbf{F}_2)$?
\end{itemize}

Formally, our main results can be summarized as follows:
\begin{thm}
    Consider a certain general scenario  in which an unknown smooth heterogeneity $\omega$ is present within a larger domain $\Omega$. In this scenario, the domain $\omega$ and its corresponding parameter configurations $\mathbf{G}$ can be uniquely determined by the boundary measurement $\mathcal{M}^{+}_{\omega,\mathbf{G}}$.
\end{thm}
\begin{thm}
    Consider a certain general scenario where there exists an anomalous smooth or non-smooth polyhedral inclusion $\omega\Subset\Omega$, such that $\mathbf{F}$ exhibits a jump discontinuity due to adherence to certain admissibility properties. In this case, $\omega$ and its parameter configurations $\mathbf{F}$ are uniquely determined by the boundary measurement $\mathcal{M}^{+}_{\omega,\mathbf{F}}$.
\end{thm}

The specifics regarding the anomalous inclusion $\omega$ and the admissible class of the unknowns $\mathbf{G}$ and $\mathbf{F}$ will be detailed in subsection \ref{sec:tp_smooth_ad} and \ref{sec:tp_nonsm_ad}. The main results of this work are contained in Theorems \ref{thm:tp_smooth_mainthm} and \ref{thm:tp_part1_mainreslut}.

%%%%%%%%%%%%%%%%%%%%%%%%%%%%%%%%%%%%%%%%%%%%%%%%%%%%%%%%%%%%%%%%%%%%%%%%%%%%%%%%%%%%%
\subsection{Technical developments and discussion}\label{sec:tp_Techn}

This work establishes a novel theoretical framework for solving inverse problems in complex  multi-species cross-diffusion predator-prey chemotaxis ecological systems, with a focus on the unique identification of internal habitat anomalies -- unknown subdomains where environmental conditions differ sharply from the surrounding region. Departing from prior studies where ecological parameters are assumed to vary smoothly or remain globally constant \cite{Ding2023inverse, LLL1}, a key contribution of this paper is the detection and characterization of anomalies that manifest as jump discontinuities in the reaction functions $\mathbf{F}$ or $\mathbf{G}$, using only exterior boundary measurements. This capability enables the non-invasive localization of degraded zones, resource hotspots, or physical barriers within a habitat -- an essential task in ecological monitoring, with direct applications ranging from invasive species management to the assessment of environmental contamination.

The mathematical detection of such internal anomalies in ecological inverse problems remains largely unexplored. While related efforts exist in fields such as source reconstruction in bioluminescence tomography \cite{Ding2023smooth} or singular geometry recovery in mean field games \cite{LL2024}, our study addresses the more general and ecologically relevant case of regions with potentially irregular (polyhedral) boundaries, which better approximate real-world habitat features. Furthermore, unlike our earlier works on nonlinear parabolic systems \cite{LLL1, LLL2025inverse}, we relax global regularity assumptions, explicitly allowing $\mathbf{F}$ and $\mathbf{G}$ to exhibit discontinuities across an unknown interface $\partial\omega$. This permits the simultaneous recovery of both the discontinuous coefficient and the geometric shape of the anomaly itself -- a considerably more challenging and ecologically meaningful problem.

Beyond handling discontinuities, our framework also delivers strong identifiability results for smoother configurations, all achieved through boundary measurements. These advances are significant for two principal reasons. First, to the best of our knowledge, this is the first work to address the shape determination of general subdomains within a multi-species chemotaxis-predator-prey model from boundary measurements. Second, we judiciously incorporate a priori ecological information by considering admissible classes of functional responses $\mathbf{G}$ that are sufficiently broad to encompass physically realistic scenarios. For smoother anomalies, we require only that the unknown subdomain possesses at least $C^2$-regular boundaries.

A distinctive feature of our model is its capacity to capture the dynamics of multiple interacting species and chemical signals, providing a more realistic representation of complex ecological networks than traditional two-species models. This multi-species, multi-chemical framework introduces substantial analytical challenges, as couplings between population densities and chemotactic gradients create a highly nonlinear and interdependent system. It is precisely within this intricate setting—where conventional simplification techniques often fail—that we successfully solve the inverse problem of determining habitat anomalies. Our results demonstrate that even under complex ecological interactions, the discontinuous topological structure of a habitat and its associated reaction functions can be uniquely and simultaneously identified from limited boundary data. This achievement underscores the robustness of our theoretical framework and significantly extends the applicability of inverse problem methodologies to realistically complex biological systems.

This advancement is not only mathematically novel but also ecologically critical, grounded in the integrative, cross-scale nature of the model. While forward modeling of chemotaxis-predator-prey dynamics is well-developed \cite{Keller1970, Lotka1925, Volterra1926, Sengupta2011, Gnanasekaran2024, Gao2024, Wang2024, Fu2020, Schaaf1985, Lin1988, Cantrell2020}, the corresponding inverse problem -- recovering internal habitat structures from boundary data -- remains largely open. In contrast to existing inverse studies in population dynamics, which have largely focused on recovering smooth coefficients under global continuity assumptions \cite{Ding2023inverse, LLL1, LLbiology1, LLL2025inverse, LLL2, LLL2025pee}, the present work pioneers a fundamentally distinct direction: the simultaneous geometric recovery of an unknown habitat anomaly $\omega$ and the discontinuous ecological parameters governing it. This represents a critical methodological shift, moving from the question of how species interact to that of where the environment itself changes abruptly. By bridging this gap, our work provides a complete mathematical framework for linking microscale environmental heterogeneity to macroscale ecological sensing via boundary measurements.

The remainder of this paper is organized as follows. In Section \ref{sec:tp_soomth_TP}, we present the main result and detailed proof for determining smooth internal anomalies and their associated coefficient functions. In Section \ref{sec:tp_nonsm_TP}, we address the non-smooth case for the stationary chemotaxis-predator-prey system.

%%%%%%%%%%%%%%%%%%%%%%%%%%%%%%%%%%%%%%%%%%%%%%%%%%%%%%%%%%%%%%%%%%%%%%%%%%%%%%%%%%%%%
%%%%%%%%%%%%%%%%%%%%%%%%%%%%%%%%%%%%%%%%%%%%%%%%%%%%%%%%%%%%%%%%%%%%%%%%%%%%%%%%%%%%%
{\centering \section{Smooth Internal Topological Structures}  \label{sec:tp_soomth_TP} }
\subsection{Admissible classes}\label{sec:tp_smooth_ad}
\begin{defi}
    Let $(\mathbf{u}_0,\mathbf{v}_0)$ be a known non-negative constant solution of \eqref{eq:tp_maingeneral}, and $E$ be a compact subset of $\Omega\subset\mathbb{R}^n$.  We say that $U(x,t,p,q): E \times \mathbb{R}\times\mathbb{C}^N  \times \mathbb{C}^M \to \mathbb{C}$ is admissible, denoted by $U \in \mathcal{B}_E,$ if:
    \begin{enumerate}[label=(\roman*)]
\item The map $(p,q) \mapsto U(\cdot,\cdot, p,q)$ is holomorphic with value in $C^{2+\alpha,1+\frac{\alpha}{2}}(E\times\mathbb{R})$ for some $\alpha\in(0,1)$,
\item $U(x,t,\cdot,\cdot)$ is $C^{\sigma_1,\sigma_2}$-continuous with respect to $(x,t)\in E\times\mathbb{R}$ for some $\sigma_1,\sigma_2 \in(0,1)$,
\item $U(x,t,\mathbf{u}_0,\mathbf{0})=0$ for all $x\in \mathbb{R}^n\times\mathbb{R}$, for any $\mathbf{u}_0 \neq 0$,
\item $U^{(0,1)}(x,t,\cdot,\mathbf{v}_0)=U^{(1,0)}(x,t,\mathbf{u}_0,\cdot)=0$ for all $(x,t)\in E\times\mathbb{R}$,
  \end{enumerate}
It is clear that if $U$ satisfies these four conditions, it can be expanded into a power series
\[U(x,t,p,q)=\sum\limits^{\infty}_{m,n=1}U_{mn}(x,t)\frac{p^m q^n}{(m+n)!}.\]
where $U_{mn}(x,t)=\frac{\partial^m}{\partial p^m}\frac{\partial^n}{\partial q^n}U(x,t,\mathbf{u}_0,\mathbf{0})\in C^{2+\alpha,1+\frac{\alpha}{2}}(E\times\mathbb{R})$.
\end{defi}

The admissibility condition is initially imposed on $U$ as a prerequisite. To fulfill this condition, we extend these functions from real variables to the complex plane, resulting in $\tilde{U}$, and assume that they possess holomorphic properties with respect to the complex variables $(p,q)$. Afterwards, $U$ is obtained by restricting $\tilde{U}$ to the real line. Moreover, since the values of $U$ lie in $\mathbb{R}$, we can assume that the series expansion of $\tilde{U}$ is real-valued. Given the well-established physical relevance of the multi-species, multi-chemical predator-prey chemotaxis model, solutions to \eqref{eq:tp_maingeneral} are assured. With this foundation, we now define the admissible class for the coefficient function $\mathbf{G}$ appearing in \eqref{eq:tp_maingeneral}.
\begin{defi}\label{def:tp_G}
   Let $\Omega$ be an open bounded set in $\mathbb{R}^n$, $n\in \mathbb{N}$, such that $\partial\Omega\in C^2$. We say that $G(x,t,\mathbf{u},\mathbf{v})\in L^2(\Omega\times \mathbb{R}\times \mathbb{R}^{N}\times \mathbb{R}^{M})$ is admissible and denoted by $G \in \mathcal{B}$, if:
    \begin{enumerate}[label=(\roman*)]
        \item  $G$ is of the form $G(x,t,\mathbf{u},\mathbf{v}) = G^0(x,t,\mathbf{u},\mathbf{v}) + (G^1(x,t,\mathbf{u},\mathbf{v})-G^0(x,t,\mathbf{u},\mathbf{v}))\chi_\omega, \  x\in\Omega,$ 
    such that $G^1\in\mathcal{B}_\omega$, $G^0\in\mathcal{B}_{\Omega\backslash\omega}$,
        \item  $\omega\Subset\Omega$ is an open bounded subset such that $\partial\omega\in C^2$, and $\Omega\backslash \omega$ is connected,
        \item $G^0(x,t,\mathbf{u},\mathbf{v})\not\equiv G^1(x,t,\mathbf{u},\mathbf{v})$ on $\partial\omega\times (0,T)$. 
    \end{enumerate}
\end{defi}

\begin{remark}
    One notable distinction is that, unlike the constraints imposed on the admissible classes for coefficients in \cite{LLL1}, we do not require their linearized counterparts to be strictly constant. This relaxation is more consistent with the physical background of the problem.
\end{remark}

%%%%%%%%%%%%%%%%%%%%%%%%%%%%%%%%%%%%%%%%%%%%%%%%%%%%%%%%%%%%%%%%%%%%%%%%%%%%%%%%%%%%%
\subsection{Main theorem}\label{sec:tp_smooth_mainthm}
For the recovery of the smooth space $\omega$, we have $(\mathbf{u}^1,\mathbf{v}^1)$ and $(\mathbf{u}^0,\mathbf{v}^0)$ satisfy:
\begin{equation}\label{eq:tp_local_smooth1}
    \begin{cases}
    \partial_t u^1_1(x,t)-d_1\Delta \left(u^1_1(x,t)+\sum_{j=1}^M\delta_{1j} u^1_1(x,t)v^1_j(x,t)\right)\\ \qquad \qquad \qquad \qquad\qquad \qquad \qquad\qquad=G^1_1(x,t,u^1_1,\dots,u^1_N,v^1_1,\dots,v^1_M)&\ \text{in } \omega\times(0,T) ,\\
    \qquad \qquad\qquad \qquad \qquad \qquad \qquad \vdots \qquad \qquad \qquad \qquad \qquad \qquad \qquad  &\quad \vdots\\   
    \partial_t u^1_N(x,t)-d_N\Delta \left(u^1_N(x,t)+\sum_{j=1}^M\delta_{Nj} u^1_N(x,t)v^1_j(x,t)\right)\\ \qquad \qquad \qquad \qquad\qquad \qquad \qquad\qquad=G^1_N(x,t,u^1_1,\dots,u^1_N,v^1_1,\dots,v^1_M) &\ \text{in } \omega\times (0,T) ,\\
    \partial_t v^1_1(x,t)-\delta_1 \Delta v^1_1(x,t)=\chi^{1}_{1}\nabla \cdot (v^1_1 \nabla u^1_1)+ \cdots + \chi^{N}_{1}\nabla \cdot (v^1_1 \nabla u^1_N)&\ \text{in } \omega\times (0,T) ,\\
    \qquad \qquad \qquad\qquad \qquad \vdots \qquad \qquad \qquad \qquad \qquad \qquad \qquad  &\quad \vdots\\   
     \partial_t v^1_M(x,t)-\delta_M \Delta v^1_M(x,t)=\chi^{1}_{M}\nabla \cdot (v^1_M \nabla u^1_1)+ \cdots + \chi^{N}_{M}\nabla \cdot (v^1_M \nabla u^1_N)&\ \text{in } \omega\times (0,T) ,\\
     %\partial_\nu u^1_1(x,t)=\cdots=\partial_\nu u^1_M(x,t)=0, \  \partial_\nu v^1_1(x,t)=\cdots=\partial_\nu v^1_N(x,t)=0 &\  \text{on } \partial\omega\times (0,T),\\
     u^1_1(x,0)=f_1, \cdots, u^1_N(x,0)=f_N,\  v^1_1(x,0)=g_1, \cdots, v^1_M(x,0)=g_M,  &\  \text{in } \omega,
    \end{cases}
\end{equation}
and
\begin{equation}\label{eq:tp_local_smooth0}
    \begin{cases}
    \partial_t u^0_1(x,t)-d_1\Delta \left(u^0_1(x,t)+\sum_{j=1}^M\delta_{1j} u^0_1(x,t)v^0_j(x,t)\right)\\ \qquad \qquad \qquad \qquad\qquad \qquad \qquad\qquad=G^0_1(x,t,u^0_1,\dots,u^0_N,v^0_1,\dots,v^0_M)&\ \text{in } (\Omega\backslash\omega)\times(0,T) ,\\
    \qquad \qquad \qquad \qquad \qquad \qquad \qquad\vdots  \qquad \qquad \qquad \qquad \qquad \qquad  &\quad \vdots\\   
    \partial_t u^0_N(x,t)-d_N\Delta \left(u^0_N(x,t)+\sum_{j=1}^M\delta_{Nj} u^0_N(x,t)v^0_j(x,t)\right)\\ \qquad \qquad \qquad \qquad\qquad \qquad \qquad\qquad=G^0_N(x,t,u^0_1,\dots,u^0_N,v^0_1,\dots,v^0_M) &\ \text{in } (\Omega\backslash\omega)\times(0,T) ,\\
    \partial_t v^0_1(x,t)-\delta_1 \Delta v^0_1(x,t)=\chi^{1}_{1}\nabla \cdot (v^0_1 \nabla u^0_1)+ \cdots + \chi^{N}_{1}\nabla \cdot (v^0_1 \nabla u^0_N)&\ \text{in } (\Omega\backslash\omega)\times(0,T) ,\\
    \qquad \qquad \qquad \qquad \qquad \vdots \qquad \qquad \qquad \qquad \qquad \qquad \qquad  &\quad \vdots\\   
     \partial_t v^0_M(x,t)-\delta_M \Delta v^0_M(x,t)=\chi^{1}_{M}\nabla \cdot (v^0_M \nabla u^0_1)+ \cdots + \chi^{N}_{M}\nabla \cdot (v^0_M \nabla u^0_N)&\ \text{in } (\Omega\backslash\omega)\times(0,T) ,\\
     %\partial_\nu u^0_1(x,t)=\cdots=\partial_\nu u^0_M(x,t)=0, \  \partial_\nu v^0_1(x,t)=\cdots=\partial_\nu v^0_N(x,t)=0 &\  \text{on } \partial\omega\times(0,T),\\
     u^0_1(x,0)=f_1, \cdots, u^0_N(x,0)=f_N,\  v^0_1(x,0)=g_1, \cdots, v^0_M(x,0)=g_M,  &\  \text{in } \Omega\backslash\omega,
    \end{cases}
\end{equation}
respectively, such that 
\begin{equation}
    \begin{cases}
        \partial_\nu u^0_i(x,t)=\partial_\nu u^1_i(x,t), \quad \partial_\nu v^0_j(x,t) = \partial_\nu v^1_j(x,t) &\quad \text{on }\partial\omega\times (0,T),\\
        u^0_i(x,0) = u^1_i(x,0), \quad v^0_j(x,0)=v^1_j(x,0) & \quad\text{on } \partial\omega\times (0,T),
    \end{cases}
\end{equation}
for $i=1,\dots,N$ and $j=1,\dots,M$.

And we present our main result as follow.
\begin{thm}\label{thm:tp_smooth_mainthm}
    Let $\omega\Subset\Omega$ be such that $\omega,\Omega$ are open bounded sets in $\mathbb{R}^n$ with $C^2$ boundary. Suppose $\mathbf{G}\in\mathcal{B}$ is such that $(\mathbf{u}^1,\mathbf{v}^1)$ and $(\mathbf{u}^0,\mathbf{v}^0)$ are classical solutions to \eqref{eq:tp_local_smooth1} and \eqref{eq:tp_local_smooth0} respectively. Then $\mathbf{G}$ is uniquely determined by the measurement $\mathcal{M}^{+}_{\omega,\mathbf{G}}$, in the sense that the smooth domain $\omega$ and coefficient function $\left.\mathbf{G}(x,t,\mathbf{u},\mathbf{v})\right|_{\partial\omega}$ are uniquely determined.
\end{thm}

%%%%%%%%%%%%%%%%%%%%%%%%%%%%%%%%%%%%%%%%%%%%%%%%%%%%%%%%%%%%%%%%%%%%%%%%%%%%%%%%%%%%%
\subsection{High-order variation method}\label{sec:tp_highvari}

Before proving Theorem \ref{thm:tp_smooth_mainthm}, we introduce the high-order linearization technique. This method enables the reduction of the nonlinear inverse problem to a sequence of linear problems while simultaneously preserving the non-negativity constraint inherent to the model. 

 Taking $\mathbf{G}\in\mathcal{B}_E$ for some compact subset $E$ of $\Omega$. Let $\mathbf{u}:(u_1,\dots,u_N)$, $\mathbf{v}:=(v_1,\dots,v_M)$. With a set of initial functions $\mathbf{u}_0$ and $\mathbf{v}_0$, along with a sufficiently small positive constant $\varepsilon$, we can express the functions $\mathbf{f}(x;\varepsilon)$ and $\mathbf{g}(x;\varepsilon)$ as follows:
\begin{equation}\notag
    \mathbf{f}(x;\varepsilon)=\mathbf{u}_{0}+\varepsilon \mathbf{f}_{1}(x)+\frac{1}{2}\varepsilon^2 \mathbf{f}_{2}(x)+ \tilde{\mathbf{f}}(x;\varepsilon),
\end{equation}
\begin{equation}\notag
    \mathbf{g}(x;\varepsilon)=\mathbf{v}_{0}+\varepsilon \mathbf{g}_{1}(x)+\frac{1}{2}\varepsilon^2 \mathbf{g}_{2}(x)+ \tilde{\mathbf{g}}(x;\varepsilon),
\end{equation}
where $\mathbf{f}_{1}, \mathbf{f}_{2}, \mathbf{g}_{1}, \mathbf{g}_{2} \in [C^{2+\alpha}(E)]^N$, and $\tilde{\mathbf{f}}(x;\epsilon), \tilde{\mathbf{g}}(x;\epsilon)$ satisfy
\[\frac{1}{|\varepsilon|^3}\norm{\tilde{\mathbf{f}}(x;\epsilon)}_{[C^{2+\alpha}(E)]^N}=\frac{1}{|\varepsilon|^3}\norm{\mathbf{f}(x;\varepsilon)-\mathbf{u}_{0}-\varepsilon \mathbf{f}_{1}(x)-\frac{1}{2}\varepsilon^2 \mathbf{f}_{2}(x)}_{[C^{2+\alpha}(E)]^N}\to0
\]
and 
\[\frac{1}{|\varepsilon|^3}\norm{\tilde{\mathbf{g}}(x;\epsilon)}_{[C^{2+\alpha}(E)]^N}=\frac{1}{|\varepsilon|^3}\norm{\mathbf{g}(x;\varepsilon)-\mathbf{v}_{0}-\varepsilon \mathbf{g}_{1}(x)-\frac{1}{2}\varepsilon^2 \mathbf{g}_{2}(x)}_{[C^{2+\alpha}(E)]^N}\to0
\]
respectively, both convergences are uniform with respect to $\varepsilon$. In the case where $\mathbf{u}_{0}=\mathbf{0}$, we impose the condition $\mathbf{f}_{1} \geq 0$, ensuring that $\mathbf{f}$ remains non-negative as $\varepsilon \to 0$. Similarly, when $\mathbf{v}_{0}=\mathbf{0}$, we require $\mathbf{g}_{1} \geq 0$.

Based on the assumption in Theorem \ref{thm:tp_smooth_mainthm}, there exists a classical solution $(\mathbf{u}(x,t;\varepsilon),\mathbf{v}(x,t;\varepsilon))$ to \eqref{eq:tp_maingeneral} in $E$. Let $S$ represent the solution operator for equation \eqref{eq:tp_maingeneral} concerning $(\mathbf{f},\mathbf{g})$. Then, there exists a bounded linear operator $A$ which maps from $[C^{2+\alpha}(E)]^{2N}$ to $[C^{1+\frac{\alpha}{2},2+\alpha}(E\times[0,T])]^{2N}$. These operators are subject to the following properties:
\begin{equation}\notag
	\lim\limits_{\norm{(\mathbf{f},\mathbf{g})}_{[C^{2+\alpha}(E)]^{2N}}\to0}\frac{\|S(\mathbf{f},\mathbf{g})-S(\mathbf{u}_0,\mathbf{v}_0)- A(\mathbf{f},\mathbf{g})\|_{[C^{1+\frac{\alpha}{2},2+\alpha}(E\times[0,T])]^{2N}}}{\norm{(\mathbf{f},\mathbf{g})}_{[C^{2+\alpha}(E)]^{2N}}}=0.
\end{equation}  

Then for fixed $\mathbf{f}_1$ and $\mathbf{g}_1$, $A\left(\mathbf{f}, \mathbf{g}\right)|_{\varepsilon=0}$  corresponds to the solution map for the first-order variation system:
\begin{equation}\label{eq:tp_smooth_Linear1}
  \begin{cases}
    \partial_{t} u^{(I)}_{i}(x,t)-d_{i}\Delta\Big( u^{(I)}_{i}(x,t) +\sum_{j=1}^M\delta_{ij} u^{(I)}_i(x,t)v_{j,0}(x,t)\\\qquad\qquad\qquad\qquad\qquad\qquad+\sum_{j=1}^M\delta_{ij} u_{i,0}(x,t)v_{j}^{(I)}(x,t)\Big)=0 & \text{ in }  E\times(0,T],\\
    \partial_{t} v^{(I)}_{j}(x,t)-\delta_{j}\Delta v^{(I)}_{j}(x,t)=\sum_{i=1}^N\chi_j^i\nabla \cdot (v_{j,0}\nabla u_i^{(I)}+v_j^{(I)}\nabla u_{i,0}) & \text{ in }  E\times(0,T],\\
   %\partial_\nu u^{(I)}_i(x,t)=\partial_\nu v^{(I)}_j(x,t)=0 & \text{ on } \partial E\times(0,T],\\
    u^{(I)}_i(x,0)=f_{1,i}(x), \  v^{(I)}_j(x,0)=g_{1,j}(x) & \text{ in } E,
  \end{cases}
\end{equation}
where $i=1,2,\dots,N$ and $j=1,2,\dots,M$ for \eqref{eq:tp_maingeneral}. Here, we define 
\begin{equation}
(\mathbf{u}^{(I)}(x,t) ,
\mathbf{v}^{(I)}(x,t)):=A(\mathbf{f},\mathbf{g})|_{\varepsilon=0}.
\end{equation}
For simplicity, we denote them as 
\begin{equation}
    u_i^{(I)}(x,t):=\partial_{\varepsilon}u_i(x,t;\varepsilon) \vert_{\varepsilon=0} ,\  v_j^{(I)}(x,t):=\partial_{\varepsilon}v_j(x,t;\varepsilon) \vert_{\varepsilon=0}.
\end{equation}
In the following discussion, we will use these notations to simplify the presentation, with their specific meanings becoming apparent in the context provided.

In particular, we observe that $(\mathbf{u}_0,\mathbf{0})$ is a solution to \eqref{eq:tp_maingeneral} for non-negative constant $\mathbf{u}_0$. Furthermore, choosing $\mathbf{g}_{1}=\mathbf{0}$, by the uniqueness of solutions to parabolic equations, it must be that $\mathbf{v}^{(I)}=\mathbf{0}$, and \eqref{eq:tp_smooth_Linear1} reduces to 
\begin{equation}\label{eq:tp_smooth_Linear1.2}
  \begin{cases}
    \partial_{t} u^{(I)}_{i}(x,t)-d_{i}\Delta u^{(I)}_{i}(x,t) =0 & \text{ in }  E\times(0,T],\\
    \partial_{t} v^{(I)}_{j}(x,t)-\delta_{j}\Delta v^{(I)}_{j}(x,t)=0 & \text{ in }  E\times(0,T],\\
   %\partial_\nu u^{(I)}_i(x,t)=\partial_\nu v^{(I)}_j(x,t)=0 & \text{ on } \partial E\times(0,T],\\
    u^{(I)}_i(x,0)=f_{1,i}(x), \  v^{(I)}_j(x,0)=0 & \text{ in } E.
  \end{cases}
\end{equation}

Next, for the second-order variation, we define
\begin{equation}\notag
    u_i^{(II)}(x,t):=\partial^{2}_{\varepsilon}u_i (x,t)\vert_{\varepsilon=0},\  v_j^{(II)}(x,t):=\partial^{2}_{\varepsilon}v_j (x,t)\vert_{\varepsilon=0}\quad\text{for }i=1,\dots,N; j=1,\dots,M.
\end{equation}
And we have the second-order variation system as follows:
\begin{equation}\label{eq:tp_smooth_Linear2}
    \begin{cases}
		\partial_t u^{(II)}_{i} -d_{i}\Delta \big[u^{(II)}_{i}+\sum\limits^{M}_{j=1} \delta_{ij} (2u^{(I)}_{i}v^{(I)}_{j}+u_{i,0}v^{(II)}_{j}+u^{(II)}_{i}v_{j,0})\big]\\
\qquad \ =\sum\limits^{N}_{h =1} \sum\limits^{N}_{\substack{k =1, \\  k\neq h}} \left[
      (G^{(II)}_{i,u_{k} u_h}+G^{(II)}_{i,u_h u_{k}}) u^{(I)}_h u^{(I)}_{k}+  G^{(II)}_{i,u_{h} u_{h}} (u^{(I)}_{h})^2 \right]&\qquad  \\
\qquad \ 
      + \sum\limits^{N}_{k =1}\sum\limits^{M}_{h =1} 
      (G^{(II)}_{i,v_{h} u_k}+G^{(II)}_{i,u_k v_{h}}) u^{(I)}_k v^{(I)}_{h}\\
\qquad \ +\sum\limits^{M}_{h =1} \sum\limits^{M}_{\substack{k =1, \\  k\neq h}} \left[
      (G^{(II)}_{i,v_{k} v_h}+G^{(II)}_{i,v_h v_{k}}) v^{(I)}_h v^{(I)}_{k}+  G^{(II)}_{i,v_{h} v_{h}} (v^{(I)}_{h})^2   \right] & \text{ in } E\times(0,T],\\
        \partial_t v^{(II)}_{j} -\delta_j \Delta v^{(II)}_{j}= \sum\limits^{N}_{i=1}\nabla\cdot [2v^{(I)}_j\nabla u^{(I)}_{i}+v^{(II)}_j\nabla u_{i,0}+v_{j,0}\nabla u^{(II)}_{i}] & \text{ in } E\times(0,T],\\
        %\partial_\nu u_i(x,t)=\partial_\nu v_j(x,t)=0 & \text{ on }\partial E\times(0,T],\\
		u^{(II)}_i(x,0)=2f_{2,i}(x), v^{(II)}_j(x,0)=2g_{2,j}(x) & \text{ in } E,\\
    \end{cases}  	
\end{equation}
where $i=1,\dots,N; j=1,\dots,M$. It is worth noting that the initial functions $\mathbf{f}_2,\mathbf{g}_2$ can be arbitrarily specified regardless of the initial value $(\mathbf{u}_0,\mathbf{v}_0)$. This flexibility arises from the fact that the non-negativity of $\mathbf{u}$ and $\mathbf{v}$ is guaranteed by the non-negativity of $\mathbf{f}_1$ and $\mathbf{g}_1$.

Once again, choosing $\mathbf{g}_2=\mathbf{0}$, the second-order variation system reduces to 
\begin{equation}\label{eq:tp_smooth_Linear2.2}
    \begin{cases}
		\partial_t u^{(II)}_{i} -d_{i}\Delta u^{(II)}_{i}=\sum\limits^{N}_{h =1} \sum\limits^{N}_{\substack{k =1, \\  k\neq h}} \left[
      (G^{(II)}_{i,u_{k} u_h}+G^{(II)}_{i,u_h u_{k}}) u^{(I)}_h u^{(I)}_{k}+  G^{(II)}_{i,u_{h} u_{h}} (u^{(I)}_{h})^2 \right] & \text{ in } E\times(0,T],\\
        \partial_t v^{(II)}_{j} -\delta_j \Delta v^{(II)}_{j}= 0 & \text{ in } E\times(0,T],\\
        %\partial_\nu u_i(x,t)=\partial_\nu v_j(x,t)=0 & \text{ on }\partial E\times(0,T],\\
		u^{(II)}_i(x,0)=2f_{2,i}(x), v^{(II)}_j(x,0)=0 & \text{ in } E,\\
    \end{cases}  	
\end{equation}
In this work, we will always assume that the Taylor coefficients of $\mathbf{G}$ are symmetric, i.e. $G^{(II)}_{i,u_{h} u_{k}}=G^{(II)}_{i,u_{k} u_{h}}$ for all $h,k=1,\dots,N$.

It is important to highlight that the non-linear terms in the system \eqref{eq:tp_smooth_Linear2} rely on the first-order linearized system \eqref{eq:tp_smooth_Linear1}. Consequently, to account for higher-order Taylor coefficients of $\mathbf{G}$, we consider the following expression for $\ell\in\mathbb{N}$:
\begin{equation}\notag
    u_i^{(\ell)}(x,t):=\partial_{\varepsilon}^{\ell}u_i(x,t)|_{\varepsilon=0},\  v_j^{(\ell)}(x,t):=\partial_{\varepsilon}^{\ell}v_j(x,t)|_{\varepsilon=0} \quad\text{for }i=1,\dots,N; j=1,\dots,M.
\end{equation}
By following this approach, we obtain a series of parabolic systems that will be utilized again in the case where there is a discontinuity in the higher-order Taylor coefficient $\mathbf{G}^{(\ell)}$.

%%%%%%%%%%%%%%%%%%%%%%%%%%%%%%%%%%%%%%%%%%%%%%%%%%%%%%%%%%%%%%%%%%%%%%%%%%%%%%%%%%%%%
\subsection{Recovery of the smooth internal domain}\label{sec:tp_smooth_domain} 
We establish the proof for smooth domains to reconstruct the unknown domains in this subsection and coefficient functions over time in the next subsection. The unique recovery results for the inverse problem \eqref{eq:tp_maingeneral} is applicable specifically to $C^2$ domains. 

Suppose, on the contrary, that there exists $\omega_1$ and $\omega_2$ such that $\omega_1\neq\omega_2$. By the regularity of the solutions $(\mathbf{u}^1,\mathbf{v}^1)$ and the associated functions of \eqref{eq:tp_local_smooth1}, and similarly for \eqref{eq:tp_local_smooth0}, we can extend the problems. We denote the extension function of $\mathbf{u}^{1}$ on $\omega_2$ as $\mathbf{u}^{1}\vert_{ext}$, extension of $\mathbf{u}^{2}$ on $\omega_1$ as $\mathbf{u}^{2}\vert_{ext}$,
     and similarly $\mathbf{u}^{0}$ on $\omega_1\cup \omega_2\backslash B_\epsilon$ as $\mathbf{u}^{0}\vert_{ext}$, for a small ball $B_\epsilon\subset\omega_1\cup \omega_2$ of radius $\epsilon$. We then define the functions as follows:
    \begin{equation}\notag
     \tilde{\mathbf{u}}^{1}=\mathbf{u}^{0}+(\mathbf{u}^{1}-\mathbf{u}^{0}\vert_{ext})\chi_{\omega_1},\ \tilde{\mathbf{u}}^{2}=\mathbf{u}^{0}+(\mathbf{u}^{2}\vert_{ext}-\mathbf{u}^{0}\vert_{ext})\chi_{\omega_1};   
    \end{equation}
    and
    \begin{equation}\notag
        \hat{\mathbf{u}}^{1}=\mathbf{u}^{0}+(\mathbf{u}^{1}\vert_{ext}-\mathbf{u}^{0}\vert_{ext})\chi_{\omega_2},\ \hat{\mathbf{u}}^{2}=\mathbf{u}^{0}+(\mathbf{u}^{2}-\mathbf{u}^{0}\vert_{ext})\chi_{\omega_2}.
    \end{equation}

Then the $\omega_i$'s are distinct in the sense that there exists two distinct coefficient functions $\mathbf{G}^1\in\mathcal{B}_{\omega_1}$ and $\mathbf{G}^2\in\mathcal{B}_{\omega_2}$ for the inverse problem \eqref{eq:tp_local_smooth1} (with associated exterior problems \eqref{eq:tp_local_smooth0} and same coefficient function $\mathbf{G}^0$) such that $\mathbf{G}^1 \neq \mathbf{G}^0$ for all $x \in \partial\omega_1 \backslash  \omega_2$ and $\mathbf{G}^2 \neq \mathbf{G}^0$ for all $x \in \partial \omega_2 \backslash \omega_1$. Without loss of generality, we take the second case and prove our theorem in $(\Omega\backslash \omega_2)\times (0,T]$. 
     By controlling the input $\mathbf{f}_1$ for each coefficient function $\mathbf{G}^{(II)}_{u_iu_i}(x,t)$, we can choose appropriate $\tilde{\mathbf{u}}^{k(I)}$ $(k=1,2)$ such that $\tilde{\mathbf{u}}^{k(I)}(x,t)=(0,\dots,1,\dots,0)$ is non-trivial only for the $i$-th species. The strategy for the proof involves confirming that every higher-order equation for $\mathbf{G}$ necessitates  $\mathbf{G}^1=\mathbf{G}^2$.  Since the first-order terms are 0, we commence with the second-order terms.  For instance, we illustrate the case where $i=1$, and the remaining $\mathbf{G}^{(II)}_{u_iu_i}$ $(i=2,\dots, N)$ can be determined utilizing a similar method. Subsequently, all other second-order terms can be deduced by extending the analysis to cases where two components of $\tilde{\mathbf{u}}^{k(I)}$ are non-trivial.

     Before we begin the proof, we recall that $\tilde{u}^{(I)}_i:=\tilde{u}^{1(I)}_i-\tilde{u}^{2(I)}_i$ for each $i=1,\dots, N$. By referring to \eqref{eq:tp_smooth_Linear1.2} and the measurement map $\mathcal{M}^{+}_{\omega,\mathbf{G}}$, it follows that $\tilde{u}^{(I)}_i$ now satisfies the following equations:
\begin{equation}\label{eq:tp_identyforuI}
    \begin{cases}
		\partial_t \tilde{u}^{(I)}_i(x,t) -d_{i}\Delta \tilde{u}^{(I)}_i(x,t)= 0 & \text{ in } Q,\\
       \partial_\nu \tilde{u}^{(I)}_i(x,t)=0 & \text{ on }\Sigma,\\
		\tilde{u}^{(I)}_i(x,0)=0 & \text{ in } \Omega.
    \end{cases}  
\end{equation}
Apparently, the only solution to \eqref{eq:tp_identyforuI} is $\tilde{u}^{(I)}_i(x,t)=0$, given the uniqueness of the solution to the heat equation. Hence, we have $\tilde{u}^{1(I)}_i(x,t)=\tilde{u}^{2(I)}_i(x,t)$.

Suppose that $G^{1(II)}_1$ and $G^{2(II)}_1$ adhere to the problem \eqref{eq:tp_local_smooth1}. We have that $\tilde{u}^{1(II)}_1$ and $\tilde{u}^{2(II)}_1$ satisfy the boundary conditions
    \[\tilde{u}^{1(II)}_1(x,t)=\tilde{u}^{2(II)}_1(x,t),\quad \partial_\nu \tilde{u}^{1(II)}_1(x,t) = \partial_\nu \tilde{u}^{2(II)}_1(x,t) \quad \text{ on } \Sigma,\] 
    and initial value condition
    \[\tilde{u}^{1(II)}_1(x,0)=\tilde{u}^{2(II)}_1(x,0) \quad \text{ in } \Omega\backslash(\omega_1\cup\omega_2).\] 
    
    Let
    \[\tilde{u}^{(II)}_1:=\tilde{u}^{1(II)}_1-\tilde{u}^{2(II)}_1,\]
    then we have
    \[\tilde{u}^{(II)}_1(x,t)|_{\Sigma} = \partial_\nu \tilde{u}^{(II)}_1(x,t)|_{\Sigma} =0 \text{ and } \tilde{u}^{(II)}_1(x,0)=0\text{ in }\Omega\backslash(\omega_1\cup\omega_2).\]
     Furthermore, $\tilde{u}^{(II)}_1(x,t)$ solves 
\[\partial_t \tilde{u}^{(II)}_1-d_1 \Delta \tilde{u}^{(II)}_1= 0 \quad \text{ in } (\Omega\backslash(\overline{\omega_1}\cup\overline{\omega_2}))\times(0,T].\]
    It is important to note that $\Omega\backslash(\omega_1\cup\omega_2)$ is a connected set. By the unique continuation principle for parabolic equations \cite{saut1987unique}, we have that $\tilde{u}^{(II)}_1(x,t)=0$ in $(\Omega\backslash(\omega_1\cup\omega_2))\times(0,T]$. Since $\tilde{u}_1^{(II)}\in H^2_{loc}(\Omega)$, by the $C^2$ regularity of the boundaries $\partial \omega_1$ and $\partial\omega_2$, it must hold that the trace of $\tilde{u}_1^{(II)}$ is 0 at the boundary $\partial(\Omega\backslash(\omega_1\cup\omega_2))\times(0,T]$.
    
   Meanwhile, $\tilde{u}^{(II)}_1(x,t)$ satisfies \eqref{eq:tp_smooth_Linear2.2}. Now consider $\partial \omega_2 \backslash \omega_1$. First, we recall the definition of $G$, whose second-order linearized form satisfies
    \begin{equation}\label{eq:tp_GExt}G^{1(II)}_{1,u_hu_k}=G^{0(II)}_{1,u_hu_k} + \left(G^{1,1(II)}_{1,u_hu_k}-\left.G^{0(II)}_{1,u_hu_k}\right\vert_{ext}\right)\chi_{\omega_1},\quad G^{2(II)}_{1,u_hu_k}=G^{0(II)}_{1,u_hu_k} + \left(G^{2,1(II)}_{1,u_hu_k}-\left.G^{0(II)}_{1,u_hu_k}\right\vert_{ext}\right)\chi_{\omega_2},\end{equation}
    for $h,k=1,\dots,N$. 
   Since $\tilde{u}^{(II)}_1=0$ on $\partial \omega_2 \backslash \omega_1$ for all $t\in(0,T]$, $\tilde{u}^{(II)}_1$ satisfies the following:
    \begin{equation}\label{eq:tp_SmoothPf1}
        0=\partial_t \tilde{u}^{(II)}_1-d_1 \Delta \tilde{u}^{(II)}_1 = G^{0(II)}_{1,u_1u_1} (\tilde{u}^{1(I)}_{1})^2 -  G^{2,1(II)}_{1,u_1u_1} (\tilde{u}^{2(I)}_{1})^2 \quad
    \text{ on } (\partial \omega_2 \backslash \omega_1)\times(0,T].
    \end{equation}
    
Substituting our initial choices of $\tilde{\mathbf{u}}^{k(I)}$, \eqref{eq:tp_SmoothPf1} simply implies that
\begin{equation}\notag
     G^{0(II)}_{1,u_{1} u_1}(x,t)=G^{2,1(II)}_{1,u_{1} u_1}(x,t)\text{ on } (\partial \omega_2 \backslash \omega_1)\times(0,T].
 \end{equation}

Yet, comparing this with the definition of $G^{2(II)}_{1,u_{1} u_1}(x,t)$ in \eqref{eq:tp_GExt} and Definition \ref{def:tp_G}, we arrive at a contradiction.
Therefore, it must be concluded that $\partial \omega_2 \backslash \omega_1=\emptyset$ and 
 \[\omega_1 = \omega_2.\]
 
%%%%%%%%%%%%%%%%%%%%%%%%%%%%%%%%%%%%%%%%%%%%%%%%%%%%%%%%%%%%%%%%%%%%%%%%%%%%%%%%%%%%%
\subsection{Recovery of the coefficient functions}\label{sec:tp_smooth_coe}
Denote $\omega:=\omega_1 = \omega_2$. Assume that $\mathbf{G}^1\neq\mathbf{G}^2$, and in particular, assume that $G^{1(II)}_{1,u_1u_1}\neq G^{2(II)}_{1,u_1u_1}$. Then, through the application of the unique continuation principle as mentioned above, we derive the following outcome:
 \begin{equation}\notag
         0=\partial_t \tilde{u}^{(II)}_1-d_1 \Delta \tilde{u}^{(II)}_1 = G^{1,1(II)}_{1,u_1u_1} (\tilde{u}^{1(I)}_{1})^2 -  G^{2,1(II)}_{1,u_1u_1} (\tilde{u}^{2(I)}_{1})^2  \text{ on } \partial \omega\times(0,T].
    \end{equation}
By substituting the settings for $\tilde{\mathbf{u}}^{k(I)}$ $ (k=1,2)$, we once again deduce that 
\[G^{1,1(II)}_{1,u_1u_1}(x,t)=G^{2,1(II)}_{1,u_1u_1}(x,t)\text{ on } \partial \omega\times(0,T].\]
 Extending the consideration beyond the domain $\omega$, we can therefore infer that
\[G^{1(II)}_{1,u_1u_1}(x,t)=G^{2(II)}_{1,u_1u_1}(x,t) \text{ in } (\overline{\Omega\backslash\omega})\times(0,T].\]

To determine the remaining second-order coefficients for $\mathbf{G}^{(II)}$, we choose the input $\mathbf{f}^{k}_1$ in a manner such that only two components in $\mathbf{\tilde{u}}^{k(I)}$ non-trivial, with the remainder set to zero. For instance, we can consider the scenario where $\mathbf{\tilde{u}}^{k(I)}=(1,1,0,\cdots,0)$ to illustrate this method. This time, $\tilde{u}^{(II)}_1$ satisfies:
\begin{equation}\notag
    \begin{split}
        0=\partial_t \tilde{u}^{(II)}_1-d_1 \Delta \tilde{u}^{(II)}_1 = 2G^{1,1(II)}_{1,u_1u_2} \tilde{u}^{1(I)}_{1}\tilde{u}^{1(I)}_{2}-  2G^{2,1(II)}_{1,u_1u_2}\tilde{u}^{2(I)}_{1}\tilde{u}^{2(I)}_{2} \quad
    \text{ on } \partial \omega\times(0,T].
    \end{split}
    \end{equation}
This indicates
\[G^{1,1(II)}_{1,u_1u_2}(x,t)=G^{2,1(II)}_{1,u_1u_2}(x,t)\text{ on } \partial \omega\times(0,T].\]
 Extending the consideration beyond the domain $\omega$, we can therefore infer that
\[G^{1(II)}_{1,u_1u_2}(x,t)=G^{2(II)}_{1,u_1u_2}(x,t) \text{ in } (\overline{\Omega\backslash\omega})\times(0,T].\]

By repeating the aforementioned process, we obtain the unique identifiability of every second-order coefficient function of $\mathbf{G}_{i}$ for $i=1,\cdots,N$, leading us to the following conclusion:
\[\mathbf{G}^{1(II)}(x,t)=\mathbf{G}^{2(II)}(x,t)\text{ in } (\overline{\Omega\backslash\omega})\times(0,T].\]

Now that all the second-order coefficient functions of $\mathbf{G}$ are known, we commence the recovery process for the third-order coefficient functions of $\mathbf{G}$. To facilitate this process, we recall $\tilde{u}^{(II)}_i:=\tilde{u}^{1(II)}_i-\tilde{u}^{2(II)}_i$ for each $i=1,\dots, N$. By referring to \eqref{eq:tp_smooth_Linear2} and the measurement map $\mathcal{M}^{+}_{\omega,\mathbf{G}}$, it follows that $\tilde{u}^{(II)}_i$ now satisfies the following equations:
\begin{equation}\label{eq:tp_Fsecond_known}
    \begin{cases}
		\partial_t \tilde{u}^{(II)}_i(x,t) -d_{i}\Delta \tilde{u}^{(II)}_i(x,t)= 0 & \text{ in } Q,\\
        \partial_\nu \tilde{u}^{(II)}_i(x,t)=0 & \text{ on }\Sigma,\\
		\tilde{u}^{(II)}_i(x,0)=0 & \text{ in } \Omega.
    \end{cases}  
\end{equation}
Apparently, the sole solution to \eqref{eq:tp_Fsecond_known} is $\tilde{u}^{(II)}_i(x,t)=0$, given the uniqueness of the solution to the heat equation. And we use $u^{(II)}_i$ to denote $u^{1(II)}_i$ and $u^{2(II)}_i$.

Then the system for the third-order variation of $u^{j(III)}_i (j=1,2)$ is once again simply a parabolic equation with the other terms involving only the lower terms $u^{(I)}_i, u^{(II)}_i$. Changing the inputs $\mathbf{f}_{1},\mathbf{f}_{2}$, we can chose $\mathbf{u}^{k(I)},\mathbf{u}^{k(II)}$ such that only the required components are non-trivial, in order to obtain the unique identifiability result for each third-order coefficient function of $\mathbf{G}_{i}$ for $i=1,\cdots,N$. This leads us to the following conclusion:
\[\mathbf{G}^{1(III)}(x,t)=\mathbf{G}^{2(III)}(x,t)\text{ in } (\overline{\Omega\backslash\omega})\times(0,T].\]

Via mathematical induction, we can extend the derived result to higher orders $\ell (\ell \geq 4)$. Consequently, we establish the unique identifiability of both the smooth domain $\omega$ and the coefficient function $\mathbf{G}$.  $ \hfill{\square}$

%%%%%%%%%%%%%%%%%%%%%%%%%%%%%%%%%%%%%%%%%%%%%%%%%%%%%%%%%%%%%%%%%%%%%%%%%%%%%%%%%%%%%
%%%%%%%%%%%%%%%%%%%%%%%%%%%%%%%%%%%%%%%%%%%%%%%%%%%%%%%%%%%%%%%%%%%%%%%%%%%%%%%%%%%%%
{\centering \section{Non-smooth Internal Topological Anomalies for the Stationary Predator-Prey Chemotaxis Model}  \label{sec:tp_nonsm_TP} }
Next, we investigates internal anomalies characterized by non-smooth boundaries. We first delineate the geometric configuration of the anomalies, then introduce the necessary admissible classes of parameters or domains, and state our main theoretical findings. The corresponding detailed proofs are deferred to subsections \ref{sec:tp_nonsm_domain} and \ref{sec:tp_nonsm_coe}.

%%%%%%%%%%%%%%%%%%%%%%%%%%%%%%%%%%%%%%%%%%%%%%%%%%%%%%%%%%%%%%%%%%%%%%%%%%%%%%%%%%%%%
\subsection{Geometrical setup}\label{sec:tp_nonsm_geo}
Assume that $\mathcal{K}_{x_c;e_1,\dots,e_l}$ is a polyhedral corner in $\Omega$ with the apex $x_c$ and edges $e_j$ $(j=1,\dots, l;l \geq n)$, where $e_j$ $(j=1,\dots, l)$ are mutually linearly independent vectors in $\Omega $. In this section, we require $\mathcal{K}_{x_c;e_1,\dots,e_l}$ to be strictly convex so that it can be enclosed within a strictly convex conic cone $\mathcal{S}_{x_c,\theta_c}$ characterized by an opening angle $\theta_c \in (0,\pi / 2)$. The conic cone $\mathcal{S}_{x_c,\theta_c}$ is defined as follows:
\[\mathcal{S}_{x_c,\theta_c}:=\{y\in \Omega:0\leq\angle (y-x_c,v_c)\leq\theta_c,\theta_c\in(0,\pi/2)\} ,\]
where $x_c$ is its apex and $v_c$ is the axis.

Given a constant $h \in \mathbb{R}_{+}$, we define the truncated polyhedral corner by
\[\mathcal{K}_h:=\mathcal{K}_{x_0;e_1,\dots,e_l}\cap B_h,\] where $B_h:=B_h(x_c)$ is an open ball contained in $\Omega$ centred at $x_c$ with radius $h>0$. Observe that both $\mathcal{K}_{x_0;e_1,\dots,e_l}$ and $\mathcal{K}_h$ are Lipschitz domains.

Let $\Gamma^{\pm}_{h}$ be two edges of $\mathcal{K}_h$, there are
\begin{equation}\label{definitely:tp_Gamma+}
\begin{split}
    \Gamma^{+}_{h}=\{x\in \Omega \vert x=r(& \cos \theta_{M_1}, \sin \theta_{M_1}\cos \theta_{M_2},\dots,\sin \theta_{M_1} \dots \sin \theta_{M_{n-2}}\cos \theta_{M_{n-1}},\\
    & \sin \theta_{M_1} \dots \sin \theta_{M_{n-2}} \sin \theta_{M_{n-1}} ) \},
\end{split}
\end{equation}
and
\begin{equation}\label{definitely:tp_Gamma-}
\begin{split}
    \Gamma^{-}_{h}=\{x\in \Omega \vert x=r(& \cos \theta_{m_1}, \sin \theta_{m_1}\cos \theta_{m_2},\dots,\sin \theta_{m_1} \dots \sin \theta_{m_{n-2}}\cos \theta_{m_{n-1}},\\
    & \sin \theta_{m_1} \dots \sin \theta_{m_{n-2}}\sin \theta_{m_{n-1}} ) \},
\end{split}
\end{equation}
 where $r \in [0,h]$, $\theta_{M_i}, \theta_{m_i} \in [0,2\pi )$ for $ i=1,\dots, n-1$ such that $\max\limits_{i} (\theta_{M_i}- \theta_{m_i}) \leq \pi$ and $\sum\limits^{n-1}_{i=1}(\theta_{M_i}- \theta_{m_i}) \leq 2\pi$. 

We begin by considering some asymptotics around this polyhedral corner for a CGO solution that we intend to use.
\begin{lem}\label{lem:tp_wCGOlem}
Let $w$ be the solution to 
\begin{equation}\label{eq:tp_CGOEq}
- \Delta w(x) = 0\text{ in }\Omega,
\end{equation}
of the form 
\begin{equation}\label{eq:tp_CGO}
w = e^{\tau (\xi + i\xi^\perp)\cdot (x-x_c)}
\end{equation} 
such that $\xi\cdot\xi^\perp=0$, $\xi,\xi^\perp\in\mathbb{S}^{n-1}$. 
Then there exists a positive number $\rho$ depending on $\mathcal{K}_h$ satisfying 
\begin{equation}\label{eq:tp_CGOCond}
-1 < \xi\cdot\widehat{(x-x_c)}\leq -\rho < 0\quad\text{ for all } x\in \mathcal{K}_h,
\end{equation}
where $\hat{x} = \frac{x}{|x|}$. Moreover, for sufficiently large $\tau$, 
\begin{equation}\label{eq:tp_CGOEst1}
\left|\int_{\mathcal{K}_h}w\right|\geq C_{\mathcal{K}_h}\tau^{-n} +\mathcal{O}\left(\frac{1}{\tau}e^{-\frac{1}{2}\rho h \tau}\right),
\end{equation}
\begin{equation}\label{eq:tp_CGOEst2}
\left|\int_{\mathcal{K}_h}|x-x_c|^\alpha w\right|\lesssim \tau^{-(\alpha+n)}+ \frac{1}{\tau}e^{-\frac{1}{2}\rho h \tau} \quad\forall \alpha>0,
\end{equation}
\begin{equation}\label{eq:tp_CGOEst3}
\norm{w}_{H^1(\partial\mathcal{K}_h)}\lesssim (2\tau^2+1)^{\frac12}e^{-\rho h\tau},
\end{equation}
\begin{equation}\label{eq:tp_CGOEst4}
\norm{\partial_\nu w}_{L^2(\partial\mathcal{K}_h)}\lesssim \tau e^{-\rho h \tau}.
\end{equation}
Here, we use the symbol `` $\lesssim$" to denote that the inequality holds up to a constant which is independent of $\tau$.
\end{lem}

\begin{proof}
    Firstly, we observe that due to the orthogonality of $\xi$ and $\xi^\perp$, the function $w$ satisfies the equation $ \Delta w = 0$ in $\mathbb{R}^n$, particularly within $\Omega$.

    Moreover, it is straightforward to fulfill condition \eqref{eq:tp_CGOCond} by selecting a suitable $\xi$ because of the properties of $\mathcal{K}_h$.

    Next, we recall from Lemma 2.2 of \cite{DiaoFeiLiuWang-2022-Semilinear-Shape-Corners} that, for some fixed $\alpha>0$ and $0<\delta<e$,
    \begin{equation}\label{eq:tp_LaplaceTransformIdentity}
        \int_0^\delta r^\alpha e^{-\mu r}\,dr = \frac{\Gamma(\alpha+1)}{\mu^{\alpha+1}} + \int_\delta^\infty r^\alpha e^{-\mu r}\,dr,
    \end{equation} 
    where $\mu\in\mathbb{C}$ and $\Gamma$ is the Gamma function. Furthermore, if the real part of $\mu$, represented as $\mathscr{R}\mu$, satisfies $\mathscr{R}\mu \geq \frac{2\alpha}{e}$, then we have the inequality $r^\alpha \leq e^{\mathscr{R}\mu r/2}$, which implies
    \begin{equation}\label{eq:tp_LaplaceTransformIneq}
        \left|\int_\delta^\infty r^\alpha e^{-\mu r}\,dr\right| \leq \frac{2}{\mathscr{R}\mu}e^{-\mathscr{R}\mu\delta/2}.
    \end{equation}
   With this, we express $x - x_c = (x_1, \dots, x_n) \in \mathbb{R}^n$ in terms of polar coordinates:
    \[x-x_c= \left( \begin{array}{l}
    r\cos\theta_1 \\
    r\sin\theta_1\cos\theta_2 \\
    r\sin\theta_1\sin\theta_2\cos\theta_3 \\
    \vdots \\
    r\sin\theta_1\cdots\sin\theta_{n-2}\cos\theta_{n-1} \\
    r\sin\theta_1\cdots\sin\theta_{n-2}\sin\theta_{n-1}
    \end{array} \right)^T,\]
    where $\theta_i\in[0,\pi]$ for $i=1,\dots,n-1$, with Jacobian
    \[r^{n-1}\sin^{n-2}(\theta_1)\sin^{n-3}(\theta_2)\cdots\sin(\theta_{n-2})\,dr\,d\theta_1 \,d\theta_2 \cdots\,d\theta_{n-2}\,d\theta_{n-1}.\]

    Next, employing the polar coordinate transformation and using \eqref{eq:tp_LaplaceTransformIdentity}, we derive the following expression:
    \begin{multline}\label{eq:tp_CGOEstProofEq}
    \int_{\mathcal{K}_h}e^{\tau (\xi + i\xi^\perp)\cdot (x-x_c)} \\ = I_1 + \int^{\theta_{M_{n-1}}}_{\theta_{m_{n-1}}} \left(\int_{\theta_{m_1}}^{\theta_{M_1}} \cdots \int_{\theta_{m_{n-2}}}^{\theta_{M_{n-2}}} I_2 \sin^{n-2}(\theta_1)\sin^{n-3}(\theta_2)\cdots\sin(\theta_{n-2})\,d\theta_1\,d\theta_2\cdots\,d\theta_{n-2} \right) \,d\theta_{n-1}, 
    \end{multline}
    where 
    \begin{multline}
        \label{eq:tp_CGOEstProofEqI1}
    I_1 := \int^{\theta_{M_{n-1}}}_{\theta_{m_{n-1}}} \int_{\theta_{m_1}}^{\theta_{M_1}} \cdots \int_{\theta_{m_{n-2}}}^{\theta_{M_{n-2}}} \frac{\Gamma(n)}{\tau^n \left((\xi + i\xi^\perp) \cdot \widehat{(x-x_c)}\right)^n} \sin^{n-2}(\theta_1)\cdots\sin(\theta_{n-2})\\\,d\theta_1\cdots\,d\theta_{n-2}\,d\theta_{n-1}\end{multline} 
    and 
    \begin{equation}\label{eq:tp_CGOEstProofEqI2}I_2 := \int_h^\infty r^{n-1} e^{\tau r (\xi + i\xi^\perp)\cdot \widehat{(x-x_c)}} \,dr.\end{equation}
    Applying the integral mean value theorem, we have that 
    \begin{align}\label{eq:tp_CGOEstProofEq1}
        I_1 & = \frac{\Gamma(n)}{\tau^n} \int^{\theta_{M_{n-1}}}_{\theta_{m_{n-1}}} \frac{1}{\left((\xi + i\xi^\perp) \cdot \widehat{(x-x_c)}(\varphi,\theta_\zeta)\right)^n} \, d\theta_{n-1} \\&\qquad\qquad\qquad\times\int_{\theta_{m_1}}^{\theta_{M_1}} \sin^{n-2}(\theta_1) \,d\theta_1 \cdots \int_{\theta_{m_{n-2}}}^{\theta_{M_{n-2}}}\sin(\theta_{n-2})\,d\theta_{n-2} \nonumber \\
        & = \frac{(\theta_{M_{n-1}}-\theta_{m_{n-1}})\Gamma(n)C_{\theta_c}}{\tau^n} \frac{1}{\left((\xi + i\xi^\perp) \cdot \widehat{(x-x_c)}(\varphi_\zeta,\theta_\zeta)\right)^n} 
    \end{align}
    where $C_{\theta_c}$ is a constant satisfying $0 < C_{\theta_c} < 1$, and $\theta_\zeta \in (\theta_{m_{n-1}},\theta_{M_{n-1}})$. Given that $\xi \cdot \widehat{(x - x_c)} > -1$ as per \eqref{eq:tp_CGOCond}, we can infer that
    \begin{equation}\label{eq:tp_CGOEstProofEq2}
    |I_1| \geq \frac{(\theta_{M_{n-1}}-\theta_{m_{n-1}})\Gamma(n)C_{\theta_c}}{\tau^n} \frac{1}{2^{n/2}}. 
    \end{equation}

    On the other hand, for sufficiently large $\tau$, by \eqref{eq:tp_LaplaceTransformIneq},
    \begin{equation}|I_2| = \left| \int_h^\infty r^{n-1} e^{r \tau (\xi + i\xi^\perp)\cdot \widehat{(x-x_c)}} \,dr \right| \leq \frac{2}{\rho\tau} e^{-\frac{1}{2}\rho h \tau}.\end{equation} 
    
    Consequently, we have that 
    \begin{multline}
    \left| \int^{\theta_{M_{n-1}}}_{\theta_{m_{n-1}}} \left(\int_{\theta_{m_1}}^{\theta_{M_1}} \cdots \int_{\theta_{m_{n-2}}}^{\theta_{M_{n-2}}} I_2 \sin^{n-2}(\theta_1)\sin^{n-3}(\theta_2)\cdots\sin(\theta_{n-2})\,d\theta_1\,d\theta_2\cdots\,d\theta_{n-2} \right) \,d\theta_{n-1} \right| 
    \\
    \leq \int^{\theta_{M_{n-1}}}_{\theta_{m_{n-1}}} \left(\int_{\theta_{m_1}}^{\theta_{M_1}} \cdots \int_{\theta_{m_{n-2}}}^{\theta_{M_{n-2}}} |I_2| \,d\theta_1\,d\theta_2\cdots\,d\theta_{n-2} \right) \,d\theta_{n-1}  \leq \frac{2\prod\limits^{n-1}_{i=1}(\theta_{M_i}-\theta_{m_i})}{\rho \tau } e^{-\frac{1}{2}\rho h \tau}.
    \end{multline}
    By combining this result with \eqref{eq:tp_CGOEstProofEq2}, we can deduce the validity of \eqref{eq:tp_CGOEst1} with $C_{\mathcal{K}_h} = \frac{(\theta_{M_{n-1}}-\theta_{m_{n-1}}) C_{\theta_c}}{2^{n/2-1}}$.

    Moving on to \eqref{eq:tp_CGOEst2}, by utilizing \eqref{eq:tp_LaplaceTransformIdentity}, we can derive higher exponents for $r$ in \eqref{eq:tp_CGOEstProofEq}, namely
    \begin{multline*}
    \int_{\mathcal{K}_h}(x-x_c)^\alpha e^{\tau (\xi + i\xi^\perp)\cdot (x-x_c)} \\ = I_{3} + \int_{\theta_{m_{n-1}}}^{\theta_{M_{n-1}}} \left(\int_{\theta_{m_1}}^{\theta_{M_1}} \cdots \int_{\theta_{m_{n-2}}}^{\theta_{M_{n-2}}} I_{4} \sin^{n-2}(\theta_1)\sin^{n-3}(\theta_2)\cdots\sin(\theta_{n-2})\,d\theta_1\,d\theta_2\cdots\,d\theta_{n-2} \right) \,d\theta_{n-1}, 
    \end{multline*}
    where 
    \begin{align*}& |I_{3}| \\
    & := \Bigg|\int_{\theta_{m_{n-1}}}^{\theta_{M_{n-1}}} \int_{\theta_{m_1}}^{\theta_{M_1}} \cdots \int_{\theta_{m_{n-2}}}^{\theta_{M_{n-2}}} \frac{\Gamma(n+\alpha)}{\tau^{n+\alpha} \left((\xi + i\xi^\perp) \cdot \widehat{(x-x_c)}\right)^{n+\alpha}} \\&\qquad\qquad\qquad\qquad\times\sin^{n-2}(\theta_1)\cdots\sin(\theta_{n-2})\,d\theta_1 \cdots\,d\theta_{n-2}\,d\theta_{n-1} \Bigg|\\
    & = \Bigg| \frac{(\theta_{M_{n-1}}-\theta_{m_{n-1}}) \Gamma(n+\alpha)C_{\theta_c}}{\tau^{n+\alpha}} \frac{1}{\left((\xi + i\xi^\perp) \cdot \widehat{(x-x_c)}(\varphi_\zeta,\theta_\zeta)\right)^{n+\alpha}} \Bigg| \\
    & \lesssim\frac{1}{\tau^{n+\alpha}} 
    \end{align*} 
    as in the procedure to obtain \eqref{eq:tp_CGOEstProofEq1}--\eqref{eq:tp_CGOEstProofEq2}, 
    and 
    \begin{equation*}|I_{4}| := \left| \int_h^\infty r^{n+\alpha-1} e^{r \tau (\xi + i\xi^\perp)\cdot \widehat{(x-x_c)}} \,dr \right| \leq \frac{2}{\rho\tau} e^{-\frac{1}{2}\rho h \tau},\end{equation*} 
    which directly leads to \eqref{eq:tp_CGOEst2}.

    Simultaneously, we have the following inequality follows from the Cauchy-Schwarz inequality, considering that $\xi\in\mathbb{S}^{n-1}$:
    \begin{equation}\norm{w}_{H^1(\partial\mathcal{K}_h)} = \left( \norm{w}_{L^2(\partial\mathcal{K}_h)}^2 + \norm{\tau (\xi + i\xi^\perp)w}_{L^2(\partial\mathcal{K}_h)}^2 \right)^{\frac{1}{2}} \leq (2\tau^2+1)^{\frac{1}{2}} \norm{w}_{L^2(\partial\mathcal{K}_h)}. \end{equation} 
   
    Moreover, when applying the polar coordinate transformation to $w$, we find that
    \begin{equation}\label{eq:tp_CGOEstProofEq3}
    \begin{split}
        \norm{w}_{L^2(\partial\mathcal{K}_h)} &= \left( \int_{\theta_{m_{n-1}}}^{\theta_{M_{n-1}}} \int_{\theta_{m_1}}^{\theta_{M_1}} \cdots \int_{\theta_{m_{n-2}}}^{\theta_{M_{n-2}}} e^{2\tau h \xi \cdot \widehat{(x-x_c)}} \,d\theta_1 \cdots \, d\theta_{n-2} \, d\theta_{n-1} \right)^{\frac{1}{2}}\\
        & \leq [ \prod\limits^{n-1}_{i=1}(\theta_{M_i}-\theta_{m_i})]^{\frac{1}{2}} e^{-\rho h\tau}
    \end{split}
    \end{equation} 
    by \eqref{eq:tp_CGOCond}. This gives \eqref{eq:tp_CGOEst3}.

    Finally, 
    \begin{equation}\norm{\partial_\nu w}_{L^2(\partial\mathcal{K}_h)} \leq \norm{\nabla w}_{L^2(\partial\mathcal{K}_h)} = \norm{\tau (\xi + i\xi^\perp)w}_{L^2(\partial\mathcal{K}_h)} \leq \sqrt{2}\tau \norm{w}_{L^2(\partial\mathcal{K}_h)} \end{equation} 
    once again by the Cauchy-Schwarz inequality.Using \eqref{eq:tp_CGOEstProofEq3} allows us to obtain the desired outcome \eqref{eq:tp_CGOEst4}.
\end{proof}

We extend the generalization of Lemma 2.1 from \cite{DiaoFeiLiuWang-2022-Semilinear-Shape-Corners} to dimensions $n \geq 2$. Our extension encompasses general corners centered at $x_c$, which may not necessarily be located at the origin. Moreover, our approach differs from Lemma 2.2 in \cite{LL2024}, in which a convex conic cones was considered rather than polyhedral cones.

%%%%%%%%%%%%%%%%%%%%%%%%%%%%%%%%%%%%%%%%%%%%%%%%%%%%%%%%%%%%%%%%%%%%%%%%%%%%%%%%%%%%%
\subsection{Admissible classes}\label{sec:tp_nonsm_ad}
\begin{defi}\label{def:tp_time_independent}
    Let $(\mathbf{u}_0,\mathbf{v}_0)$ be a known non-negative constant solution of \eqref{eq:tp_mainstationary}, and $E$ be a compact subset of $\Omega\subset\mathbb{R}^n$.  We say that $U(x,p,q): E \times \mathbb{C}^N  \times \mathbb{C}^M \to \mathbb{C}$ is admissible, denoted by $U \in \mathcal{A}_E,$ if:
   \begin{enumerate}[label=(\roman*)]
\item The map $(p,q) \mapsto U(\cdot, p,q)$ is holomorphic with value in $C^{2+\alpha}(E)$ for some $\alpha\in(0,1)$,
\item $U(x,\cdot,\cdot)$ is $C^{\sigma_1,\sigma_2}$-continuous with respect to $x\in E$ for some $\sigma_1,\sigma_2 \in(0,1)$,
\item  $U(x,\mathbf{u}_0,\mathbf{0})=0$ for all $x\in \mathbb{R}^n$, for any $\mathbf{u}_0 \neq 0$,
\item  $U^{(0,1)}(x,\cdot,\mathbf{v}_0)=U^{(1,0)}(x,\mathbf{u}_0,\cdot)=0$ for all $x \in E,$ 
\end{enumerate}
It is clear that if $U$ satisfies these four conditions, it can be expanded into a power series
\[U(x,p,q)=\sum\limits^{\infty}_{m,n=1}U_{mn}(x)\frac{p^m q^n}{(m+n)!}.\]
where $U_{mn}(x)=\frac{\partial^m}{\partial p^m}\frac{\partial^n}{\partial q^n}U(x,\mathbf{u}_0,\mathbf{0})\in C^{2+\alpha}(E)$.
\end{defi}

Once again, the admissibility condition is initially imposed on $U$ as a prerequisite. To fulfill this condition, we extend these functions from real variables to the complex plane, resulting in $\tilde{U}$, and assume that they possess holomorphic properties with respect to the complex variables $(p,q)$. Afterwards, $U$ is obtained by restricting $\tilde{U}$ to the real line. Moreover, since the values of $U$ lie in $\mathbb{R}$, we can assume that the series expansion of $\tilde{U}$ is real-valued.

Consider a bounded Lipschitz polyhedral domain $\omega\Subset\Omega$ satisfying the condition that $\Omega\backslash\bar{\omega}$ is connected, and each corner of $\omega$ is a convex cone $\mathcal{S}_{x_c,\theta_c}$, as defined in Subsection \ref{sec:tp_nonsm_geo}. Due to the strong physical significance of the multi-species-multi-chemical predator-prey chemotaxis model, solutions exist for \eqref{eq:tp_mainstationary}. Consequently, we can determine the admissibility class for the coefficient function $\mathbf{F}$.
\begin{defi}\label{def:tp_F}
    We say that $F(x,u_1,\dots,u_N,v_1,\dots,v_M):\Omega\times\mathbb{R}^{N+M}\to\mathbb{R}$ is admissible, denoted by $F\in\mathcal{A}$, if 
    $F$ is of the form 
    \[F(x,\mathbf{u},\mathbf{v}) = F^0(x,\mathbf{u},\mathbf{v}) + (F^1(x,\mathbf{u},\mathbf{v})-F^0(x,\mathbf{u},\mathbf{v}))\chi_\omega, \quad x\in\Omega,\]
    such that $F^1\in\mathcal{A}_\omega$, $F^0\in\mathcal{A}_{\Omega\backslash\omega}$, and for some $\ell\in\mathbb{N}$, the $\ell$-th Taylor coefficient of $F^1$ and $F^0$, denoted by $F^{1(\ell)}(x,\mathbf{u},\mathbf{v})$ and $F^{0(\ell)}(x,\mathbf{u},\mathbf{v})$ respectively, are $C^{\sigma_1}$ and $C^{\sigma_0}$ H\"older-continuous for some $\sigma_1, \sigma_0 \in(0,1)$ with respect to $x\in \omega$ and $x\in \Omega\backslash \omega$ respectively, and in an open neighbourhood $U$ of $\partial \omega$, $F^{1(\ell)}(U\cap \omega) \neq F^{0(\ell)}(U\cap(\Omega\backslash \omega))$.
\end{defi}

It is worth noting that the admissibility condition we consider differs from those in previous works on biology models, such as \cite{LLL1,LLL2,LLbiology1,Ding2023inverse}, where the authors assume that the coefficient function $F$ is analytic in the entire real space $\mathbb{R}^n$. In our case, we adopt a more general approach, allowing for singularities at the boundary $\partial E$ of compact subsets $E$ of $\Omega$.

%%%%%%%%%%%%%%%%%%%%%%%%%%%%%%%%%%%%%%%%%%%%%%%%%%%%%%%%%%%%%%%%%%%%%%%%%%%%%%%%%%%%%
\subsection{Main theorem}\label{sec:tp_nonsm_mainthm}
Having established these definitions, we now present the principal results of this subsection. Specifically, we demonstrate that, microlocally within a corner $\mathcal{K}_h$ of the domain $\omega$, the solution pairs $(\mathbf{u}^1,\mathbf{v}^1)$ and $(\mathbf{u}^0,\mathbf{v}^0)$ in the habitat adopt the forms as:
\begin{equation}\label{eq:tp_micro1}
    \begin{cases}    
    -d_1\Delta\left(u^1_1(x)+\sum_{j=1}^M\delta_{1j} u^1_1(x)v^1_j(x)\right)=F^1_1(x,u^1_1,\dots,u^1_N,v^1_1,\dots,v^1_M) &\quad \text{in } \mathcal{K}_h ,\\
    \qquad \qquad \qquad  \qquad \qquad \qquad \qquad \qquad \vdots &\quad \vdots\\   
    -d_N\Delta \left(u^1_N(x)+\sum_{j=1}^M\delta_{Nj} u^1_N(x)v^1_j(x)\right)=F_N(x,u^1_1,\dots,u^1_N,v^1_1,\dots,v^1_M) &\quad \text{in } \mathcal{K}_h ,\\ 
    -\delta_1 \Delta v^1_1(x)=\chi^{1}_{1}\nabla \cdot (v^1_1 \nabla u^1_1)+ \cdots + \chi^{N}_{1}\nabla \cdot (v^1_1 \nabla u^1_N)&\quad \text{in } \mathcal{K}_h ,\\
    \qquad \qquad \qquad \vdots \qquad \qquad \qquad \qquad \qquad \qquad  &\quad \vdots\\   
     -\delta_M \Delta v^1_M(x)=\chi^{1}_{M}\nabla \cdot (v^1_M \nabla u^1_1)+ \cdots + \chi^{N}_{M}\nabla \cdot (v^1_M \nabla u^1_N)&\quad \text{in } \mathcal{K}_h ,\\
     u^1_1(x)=f_1, \cdots, u^1_N(x)=f_N, \ v^1_1(x)=g_1, \cdots, v^1_M(x)=g_M,  &\quad \text{on } \partial \Omega,
    \end{cases}
\end{equation}
and
\begin{equation}\label{eq:tp_micro2}
    \begin{cases}    
    -d_1\Delta\left(u^0_1(x)+\sum_{j=1}^M\delta_{1j} u^0_1(x)v^1_j(x)\right)=F^0_1(x,u^0_1,\dots,u^0_N,v^0_1,\dots,v^0_M) &\quad \text{in } \mathcal{B}_h\backslash \mathcal{K}_h ,\\
    \qquad \qquad \qquad  \qquad \qquad \qquad \qquad \qquad \vdots &\quad \vdots\\   
    -d_N\Delta \left(u^0_N(x)+\sum_{j=1}^M\delta_{Nj} u^0_N(x)v^0_j(x)\right)=F^0_N(x,u^0_1,\dots,u^0_N,v^0_1,\dots,v^0_M) &\quad \text{in } \mathcal{B}_h\backslash \mathcal{K}_h ,\\ 
    -\delta_1 \Delta v^0_1(x)=\chi^{1}_{1}\nabla \cdot (v^0_1 \nabla u^0_1)+ \cdots + \chi^{N}_{1}\nabla \cdot (v^0_1 \nabla u^0_N)&\quad \text{in } \mathcal{B}_h\backslash \mathcal{K}_h ,\\
    \qquad \qquad \qquad \vdots \qquad \qquad \qquad \qquad \qquad \qquad  &\quad \vdots\\   
     -\delta_M \Delta v^0_M(x)=\chi^{1}_{M}\nabla \cdot (v^0_M \nabla u^0_1)+ \cdots + \chi^{N}_{M}\nabla \cdot (v^0_M \nabla u^0_N)&\quad \text{in } \mathcal{B}_h\backslash \mathcal{K}_h ,\\
     u^0_1(x)=f_1, \cdots, u^0_N(x)=f_N, \ v^0_1(x)=g_1, \cdots, v^0_M(x)=g_M,  &\quad \text{on } \partial \Omega,
    \end{cases}
\end{equation}
respectively, such that 
\begin{equation}
        u^0_i(x) = u^1_i(x), \quad v^0_j(x)=v^1_j(x) \quad\text{on } \partial\mathcal{K}_h\backslash\partial B_h,
\end{equation}
for $i=1,\dots,N$ and $j=1,\dots,M$.
 
We next state our main results.

\begin{thm}\label{thm:tp_part1_mainreslut}
    Suppose that $\mathbf{F}\in\mathcal{A}$ and $\omega\Subset\Omega$ is a convex polyhedron. Assume that $(\mathbf{u}^1,\mathbf{v}^1)$ and $(\mathbf{u}^0,\mathbf{v}^0)$ are classical solutions to \eqref{eq:tp_micro1} and \eqref{eq:tp_micro2} respectively. Then $\omega$ and $\mathbf{F}$ are uniquely determined by the boundary measurement $\mathcal{M}^{+}_{\omega,\mathbf{F}}$. 
\end{thm}

Drawing from this theorem, we derive two corollaries. When each corner within the polyhedral domain $\omega$ takes the form of a convex conic cone $\mathcal{S}\subset\Omega$, characterized by
\[\mathcal{S}:=\{y\in\Omega:0\leq\angle (y-x_c,v_c)\leq\theta_c,\theta_c\in(0,\pi/2)\},\] 
where $x_c$ represents the apex, $v_c$ signifies the axis, and the opening angle is $2\theta_c\in(0,\pi)$. By defining a new admissible class $\mathcal{C}$ akin to Definition \ref{def:tp_F}, Corollary \ref{co:tp_change_area} affirms the following:

\begin{cor}\label{co:tp_change_area}
    Suppose that $\mathbf{F}\in\mathcal{C}$  and $\omega\Subset\Omega$ is a convex domain with conic corners. Assume that $(\mathbf{u}^1,\mathbf{v}^1)$ and $(\mathbf{u}^0,\mathbf{v}^0)$ are classical solutions to \eqref{eq:tp_micro1} and \eqref{eq:tp_micro2} respectively. Then $\omega$ and $\mathbf{F}$ are uniquely determined by the boundary measurement $\mathcal{M}^{+}_{\omega,\mathbf{F}}$. 
\end{cor}

The verification for this assertion follows as in the proof of the main theorem, by replacing Lemma \ref{lem:tp_wCGOlem} with the corresponding one for conic corners (the proof of which can be found in Lemma 2.2 of \cite{LL2024}).

The second corollary asserts that if $\omega\Subset\Omega$ is a smooth domain, the result follows directly from the time-dependent case and it is feasible to retrieve the anomaly and coefficient functions:
\begin{cor}
    Suppose that $\mathbf{F}\in\mathcal{A}$.  Let $\omega\Subset\Omega$ be such that $\omega,\Omega$ are open bounded sets in $\mathbb{R}^n$ with $C^2$ boundary. Assume that $(\mathbf{u}^1,\mathbf{v}^1)$ and $(\mathbf{u}^0,\mathbf{v}^0)$ are classical solutions to \eqref{eq:tp_micro1} and \eqref{eq:tp_micro2} respectively. Then the smooth domain $\omega$ and the coefficient function $\mathbf{F}$ are uniquely determined by the boundary measurement $\mathcal{M}^{+}_{\omega,\mathbf{F}}$. 
\end{cor}

This is a direct result of Theorem \ref{thm:tp_smooth_mainthm}, since the dynamic establishment naturally extends to the corresponding equilibrium state under the assumption of Theorem \ref{thm:tp_smooth_mainthm}.

One intriguing observation is the presence of two distinct scenarios when $\mathbf{F}=0$ within $\omega$ or when $\mathbf{F}=0$ outside of $\omega$. In the former case, $\omega$ signifies habitat degradation where no species can survive. In contrast, in the latter case, $\omega$ may represent an enclosed area, suggesting a high likelihood of animals being confined or microorganisms akin to those in a petri dish.

%%%%%%%%%%%%%%%%%%%%%%%%%%%%%%%%%%%%%%%%%%%%%%%%%%%%%%%%%%%%%%%%%%%%%%%%%%%%%%%%%%%%%
\subsection{Recovery of the polyhedral domain}\label{sec:tp_nonsm_domain}
In this subsection, we prove Theorem \ref{thm:tp_part1_mainreslut}, beginning with the recovery of the polyhedral domain. First, we introduce a key auxiliary lemma; its proof follows that of Theorem 4.1 in \cite{LL2024}, adapted to the case of polyhedral corners.

\begin{lem}\label{lem:tp_AuxThm}
    Let $\omega\Subset\Omega$ be the bounded Lipschitz domain such that $\Omega\backslash\bar{\omega}$ is connected, with sectorial corner $\mathcal{K}_h$. For $P\in L^2(\Omega)$, suppose that $q_1,q_0$ are $C^\gamma$ H\"older-continuous for some $\gamma\in(0,1)$ with respect to $x\in \omega$, such that $q_1\neq q_0$ in $\mathcal{K}_h$. For any function $h\in C^1(\Omega)$, consider the following system of equations for $w_m\in C^{2+\alpha}(\mathcal{K}_h)$, $m=1,0$:
    \begin{equation}\label{eq:tp_AuxThmEq}
    \begin{cases}
		-D \Delta w_m(x)+ P + q_m(x)h(x)= 0& \quad \text{ in } \mathcal{K}_h,\\
  \partial w_1(x)=\partial w_0(x) &\quad \text{ on }\partial\mathcal{K}_h\backslash\partial B_h,\\
        w_1(x)=w_0(x) &\quad \text{ on }\partial\mathcal{K}_h\backslash\partial B_h.
    \end{cases}  	
\end{equation}
Then, it must hold that \[h(x_c) =0,\] where $x_c$ is the apex of $\mathcal{K}_h$.
\end{lem}

We now proceed with the proof of Theorem \ref{thm:tp_part1_mainreslut}. Arguing by contradiction, there exist two distinct polyhedral inclusions, $\omega_1 \neq \omega_2$. Since both are polyhedral, there exists a corner $\mathcal{K}$ of $\omega_1$ such that $\mathcal{K} \Subset \Omega \backslash \overline{\omega_2}$.
Corresponding to each domain $\omega_m$ ($m=1,2$), let $\mathbf{F} \in \mathcal{A}$ with forcing terms $\mathbf{F}^0$ and $\mathbf{F}^m$, respectively. Following the approach in Section \ref{sec:tp_smooth_domain}, we define the following extended functions: $\mathbf{u}^{1}$ on $\omega_2$ as $\mathbf{u}^{1}\vert_{ext}$, extension of $\mathbf{u}^{2}$ on $\omega_1$ as $\mathbf{u}^{2}\vert_{ext}$, and $\mathbf{u}^{0}$ on $(\omega_1\cup \omega_2)\backslash B_\epsilon$ as $\mathbf{u}^{0}\vert_{ext}$, for a small ball $B_\epsilon\subset\omega_1\cup \omega_2$ of radius $\epsilon$. We then define the functions as follows:
    \begin{equation}\label{eq:tp_tildeEXT2}
     \tilde{\mathbf{u}}^{1}=\mathbf{u}^{0}+(\mathbf{u}^{1}-\mathbf{u}^{0}\vert_{ext})\chi_{\omega_1},\ \tilde{\mathbf{u}}^{2}=\mathbf{u}^{0}+(\mathbf{u}^{2}\vert_{ext}-\mathbf{u}^{0}\vert_{ext})\chi_{\omega_1};   
    \end{equation}
    and
    \begin{equation}\label{eq:tp_hatEXT2}
        \hat{\mathbf{u}}^{1}=\mathbf{u}^{0}+(\mathbf{u}^{1}\vert_{ext}-\mathbf{u}^{0}\vert_{ext})\chi_{\omega_2},\ \hat{\mathbf{u}}^{2}=\mathbf{u}^{0}+(\mathbf{u}^{2}-\mathbf{u}^{0}\vert_{ext})\chi_{\omega_2}.
    \end{equation}
%We also define $\tilde{\mathbf{v}}^{k}$ and $\hat{\mathbf{v}}^m$ $(m=1,2)$ in the same way. 
And the coefficient function $\mathbf{F}$ at any order $k=1,2,\dots$ satisfies:
\begin{equation}\label{eq:tp_FExt}F^{1(k)}_{i}=F^{0(k)}_{i} + \left(F^{1,1(k)}_{i}-\left.F^{0(k)}_{i}\right\vert_{ext}\right)\chi_{\omega_1},\quad F^{2(k)}_{i}=F^{0(k)}_{i} + \left(F^{2,1(k)}_{i}-\left.F^{0(k)}_{i}\right\vert_{ext}\right)\chi_{\omega_2},\end{equation}

For $\mathbf{F}^{m,1}\in\mathcal{A}_{\omega_{m}}$ ($m=1,2$) and $\mathbf{F}^0 \in\mathcal{A}_{(\Omega\backslash\omega_{1})\cup(\Omega\backslash\omega_{2})}$, we can conduct higher-order variation separately for $(\mathbf{u}^1,\mathbf{v}^1)$ and $(\mathbf{u}^0,\mathbf{v}^0)$ in $\mathcal{K}_h$ and $B_h\backslash\mathcal{K}_h$ respectively. Taking $\mathbf{F}\in\mathcal{A}_E$ for some compact subset $E$ of $\Omega$, the expansions for solutions to \eqref{eq:tp_mainstationary} are defined as follows:
\begin{equation}\notag
    \mathbf{f}(x;\varepsilon)=\mathbf{u}_{0}+\varepsilon \mathbf{f}_{1}(x)+\frac{1}{2}\varepsilon^2 \mathbf{f}_{2}(x)+ \tilde{\mathbf{f}}(x;\varepsilon),
\end{equation}
\begin{equation}\notag
    \mathbf{g}(x;\varepsilon)=\mathbf{v}_{0}+\varepsilon \mathbf{g}_{1}(x)+\frac{1}{2}\varepsilon^2 \mathbf{g}_{2}(x)+ \tilde{\mathbf{g}}(x;\varepsilon),
\end{equation}
where $\mathbf{f}_{1}, \mathbf{f}_{2}, \mathbf{g}_{1}, \mathbf{g}_{2} \in [C^{2+\alpha}(E)]^N$, and $\tilde{\mathbf{f}}(x;\epsilon), \tilde{\mathbf{g}}(x;\epsilon)$ satisfy
\[\frac{1}{|\varepsilon|^3}\norm{\tilde{\mathbf{f}}(x;\epsilon)}_{[C^{2+\alpha}(E)]^N}=\frac{1}{|\varepsilon|^3}\norm{\mathbf{f}(x;\varepsilon)-\mathbf{u}_{0}-\varepsilon \mathbf{f}_{1}(x)-\frac{1}{2}\varepsilon^2 \mathbf{f}_{2}(x)}_{[C^{2+\alpha}(E)]^N}\to0
\]
and 
\[\frac{1}{|\varepsilon|^3}\norm{\tilde{\mathbf{g}}(x;\epsilon)}_{[C^{2+\alpha}(E)]^N}=\frac{1}{|\varepsilon|^3}\norm{\mathbf{g}(x;\varepsilon)-\mathbf{v}_{0}-\varepsilon \mathbf{g}_{1}(x)-\frac{1}{2}\varepsilon^2 \mathbf{g}_{2}(x)}_{[C^{2+\alpha}(E)]^N}\to0
\]
respectively, and both the convergences are uniformly in $\varepsilon$. 
When $\mathbf{u}_{0}=\mathbf{0},$ we ask $\mathbf{f}_{1} \geq 0,$ thus $\mathbf{f}$ maintains non-negativity as $\varepsilon \to 0.$ Similarly, for $\mathbf{v}_{0}=\mathbf{0},$ we ask $\mathbf{g}_{1} \geq 0$.

Then the first-order variation system is given by:
\begin{equation}\label{eq:tp_space_Linear1}
  \begin{cases}
    -d_{i}\Delta\left( u^{(I)}_{i}(x) +\sum_{j=1}^M\delta_{ij} u^{(I)}_i(x)v_{j,0}(x)+\sum_{j=1}^M\delta_{ij} u_{i,0}(x)v_{j}^{(I)}(x)\right) =0 & \text{ in }  E,\\
    -\delta_{j}\Delta v^{(I)}_{j}(x)=\sum_{i=1}^N\chi_j^i\nabla \cdot (v_{j,0}\nabla u_i^{(I)}+v_j^{(I)}\nabla u_{i,0}) & \text{ in }  E,\\
    u^{(I)}_i(x)=f_{1,i}(x), \  v^{(I)}_j(x)=g_{1,j}(x) & \text{ on } \partial E,
  \end{cases}
\end{equation}
where $i=1,2,\dots,N$ and $j=1,2,\dots,M$ for \eqref{eq:tp_mainstationary}.

Consider the base solution $(\mathbf{u}_0,\mathbf{0})$ to \eqref{eq:tp_mainstationary}, where $\mathbf{u}_0$ is a non-negative constant vector. Choosing $\mathbf{g}_{1}=\mathbf{0}$, the uniqueness of solutions to elliptic equations implies $\mathbf{v}^{(I)}=\mathbf{0}$. Consequently, \eqref{eq:tp_space_Linear1} simplifies to
\begin{equation}\label{eq:tp_space_Linear1.2}
  \begin{cases}
    -d_{i}\Delta u^{(I)}_{i}(x)  =0 & \text{ in }  E,\\
    -\delta_{j}\Delta v^{(I)}_{j}(x)=0 & \text{ in }  E,\\
    u^{(I)}_i(x)=f_{1,i}(x), \  v^{(I)}_j(x)=0 & \text{ on } \partial E.
  \end{cases}
\end{equation}

Therefore, within the neighborhood $B_h$ of the corner $\mathcal{K}_h$, the first-order terms $u^{1(I)}_i$ and $u^{0(I)}_i$ must satisfy the following transmission problem:
\begin{equation}\notag
    \begin{cases}
		 -d_i \Delta u^{1(I)}_i(x)= 0 & \quad\text{ in }  \mathcal{K}_h,\\
         -d_i \Delta u^{0(I)}_i(x)= 0 & \quad\text{ in }  B_h\backslash\mathcal{K}_h,\\
         \partial_{\nu} u^{1(I)}_{i}(x)=\partial_{\nu} u^{0(I)}_{i}(x) &\quad \text{ on }\partial\mathcal{K}_h\backslash\partial B_h,\\
         u^{1(I)}_{i}(x)=u^{0(I)}_{i}(x) &\quad \text{ on }\partial\mathcal{K}_h\backslash\partial B_h,\\
         u^{1(I)}_i|_{ext}(x)=u^{0(I)}_i(x) = f_{1,i}(x) &\quad \text{ on }\partial\Omega,
    \end{cases}
\end{equation}
when $\mathcal{M}_{\omega_1,\mathbf{F}^1}=\mathcal{M}_{\omega_2,\mathbf{F}^2}$. Hence, it is clear that $u^{1(I)}_i(x) = u^{0(I)}_i(x)$ on $\partial\omega_1$ for each $i$. Since $u^{1(I)}_i, u^{0(I)}_i$ are both $C^{2+\alpha}$ functions and coincide on the boundary, $\tilde{u}^{1(I)}_i\in C^{2+\alpha}(\Omega)$.

Next, we consider the case of $\ell=2$. Choosing $\mathbf{g}_2=\mathbf{0}$, the second-order variation system for \eqref{eq:tp_mainstationary} becomes: 
\begin{equation}\notag
    \begin{cases}
		-d_{i}\Delta u^{(II)}_{i}%+\sum\limits^{M}_{j=1} \delta_{ij} (2u^{(I)}_{i}v^{(I)}_{j}+u_{i,0}v^{(II)}_{j}+u^{(II)}_{i}v_{j,0})\big]\\
%\qquad \qquad
=\sum\limits^{N}_{h =1} \sum\limits^{N}_{\substack{k =1, \\  k\neq h}} \left[
      (F^{(II)}_{i,u_{k} u_h}+F^{(II)}_{i,u_h u_{k}}) u^{(I)}_h u^{(I)}_{k}+  F^{(II)}_{i,u_{h} u_{h}} (u^{(I)}_{h})^2 \right]%&\qquad  \\
% \qquad \qquad\qquad
%       + \sum\limits^{N}_{k =1}\sum\limits^{M}_{h =1} 
%       (F^{(II)}_{i,v_{h} u_k}+F^{(II)}_{i,u_k v_{h}}) u^{(I)}_k v^{(I)}_{h}\\
% \qquad \qquad\qquad
% +\sum\limits^{M}_{h =1} \sum\limits^{M}_{\substack{k =1, \\  k\neq h}} \left[
%       (F^{(II)}_{i,v_{k} v_h}+F^{(II)}_{i,v_h v_{k}}) v^{(I)}_h v^{(I)}_{k}+  F^{(II)}_{i,v_{h} v_{h}} (v^{(I)}_{h})^2   \right] 
& \text{ in } E,\\
      -\delta_j \Delta v^{(II)}_{j}=0%\sum\limits^{N}_{i=1}\nabla\cdot [2v^{(I)}_j\nabla u^{(I)}_{i}+v^{(II)}_j\nabla u_{i,0}+v_{j,0}\nabla u^{(II)}_{i}] 
      & \text{ in } E,\\
		u^{(II)}_i=2f_{2,i},\  v^{(II)}_j=0%2g_{2,j} 
  & \text{ on } \partial E.\\
    \end{cases}  	
\end{equation}

Now $u^{1(II)}_i, u^{0(II)}_i$ satisfy:
\begin{equation}\label{eq:tp_sec4_u2}
    \begin{cases}
		-d_i \Delta u^{1(II)}_i-\sum\limits^{N}_{h =1} \sum\limits^{N}_{\substack{k =1, \\  k\neq h}} \big[
      (F^{1,1(II)}_{i,u_{k} u_h}+F^{1,1(II)}_{i,u_h u_{k}}) \tilde{u}^{1(I)}_{i,h} \tilde{u}^{1(I)}_{i,k}+  F^{1,1(II)}_{i,u_{h} u_{h}} (\tilde{u}^{1(I)}_{i,h})^2   \big]=0 & \quad \text{ in } \mathcal{K}_h,\\
        -d_i \Delta u^{0(II)}_i-\sum\limits^{N}_{h =1} \sum\limits^{N}_{\substack{k =1, \\  k\neq h}} \big[
      (F^{0(II)}_{i,u_{k} u_h}+F^{0(II)}_{i,u_h u_{k}}) \tilde{u}^{1(I)}_{i,h} \tilde{u}^{1(I)}_{i,k}+  F^{0(II)}_{i,u_{h} u_{h}} (\tilde{u}^{1(I)}_{i,h})^2  \big]=0 & \quad \text{ in } B_h\backslash\mathcal{K}_h,\\
        u_i^{1(II)}(x)=u_i^{0(II)}(x),\quad  \partial_{\nu} u^{1(II)}_{i}(x)=\partial_{\nu} u^{0(II)}_{i}(x)  &\quad \text{ on }\partial\mathcal{K}_h\backslash\partial B_h.
    \end{cases}  	
\end{equation}
The cross-diffusion terms vanishes as we chose the boundary function $\mathbf{g}$ as $\mathbf{0}$.

Notice that because $u^{1(I)}_i(x)$ is smooth and already fixed from our previous discussion for the first order linearization, \eqref{eq:tp_sec4_u2} is simply an elliptic equation. Therefore, we can apply the unique continuation principle for elliptic equations \cite{koch2001carleman} to obtain
\begin{equation}\label{eq:tp_Uniform_proofF2}
    \begin{cases}
		-d_i \Delta u^{1(II)}_i-\sum\limits^{N}_{h =1} \sum\limits^{N}_{\substack{k =1, \\  k\neq h}} \big[
      (F^{1,1(II)}_{i,u_{k} u_h}+F^{1,1(II)}_{i,u_h u_{k}}) u^{1(I)}_{i,h} u^{1(I)}_{i,k}+  F^{1,1(II)}_{i,u_{h} u_{h}} (u^{1(I)}_{i,h})^2   \big]=0 & \quad \text{ in } \mathcal{K}_h,\\
        -d_i \Delta u^{0(II)}_i-\sum\limits^{N}_{h =1} \sum\limits^{N}_{\substack{k =1, \\  k\neq h}} \big[
      (F^{0(II)}_{i,u_{k} u_h}+F^{0(II)}_{i,u_h u_{k}}) u^{1(I)}_{i,h} u^{1(I)}_{i,k}+  F^{0(II)}_{i,u_{h} u_{h}} (u^{1(I)}_{i,h})^2  \big]=0 & \quad \text{ in } \mathcal{K}_h,\\
        u_i^{1(II)}(x)=u_i^{0(II)}(x),\quad\  \partial_{\nu} u^{1(II)}_{i}(x)=\partial_{\nu} u^{0(II)}_{i}(x) &\quad \text{ on }\partial\mathcal{K}_h\backslash\partial B_h.
    \end{cases}  	
\end{equation}

Similar to subsection \ref{sec:tp_smooth_domain}, we provide a different input $\mathbf{f}_1$ to choose appropriate $\mathbf{u}^{1(I)}$. We begin with the case where $\mathbf{f}_1=(1,0,\cdots,0)$. Then, by the uniqueness of solutions to elliptic equations, it must be that $\mathbf{u}^{1(I)}(x)=(1,0,\cdots,0)$, which has a non-trivial element only in the first position, if $F^{m(II)}_{u_1u_1}$ has a jump singularity. And we can apply the same method to address the case where $F^{m(II)}_{u_iu_i}$ is unknown for $i=2,\dots,N$.

Under this assumption, \eqref{eq:tp_Uniform_proofF2} becomes:
\begin{equation}\label{eq:tp_case_recoverF11}
    \begin{cases}
		-d_1 \Delta u^{1(II)}_1-F^{1(II)}_{1,u_1u_1}(u^{1(I)})^2 = -d_1 \Delta u^{0(II)}_1-F^{0(II)}_{1,u_1u_1}(u^{1(I)})^2=0 &\quad \text{ in } \mathcal{K}_h,\\
        u_1^{1(II)}(x)=u_1^{0(II)}(x) &\quad \text{ on }\partial\mathcal{K}_h\backslash\partial B_h.
    \end{cases} 
\end{equation}

Denote $w_n(x)=u^{n(II)}_1(x)(n=0,1) $, $P=0$, $q_n=F^{n(II)}_{1,u_1 u_1}$ and $h(x)=(u^{1(I)}_1)^2(x)$, \eqref{eq:tp_case_recoverF11} can be transformed into \eqref{eq:tp_AuxThmEq}. If 
\[F^{1,1(II)}_{1,u_1u_1}(x_c) \neq F^{0(II)}_{1,u_1u_1}(x_c), \]
there must be $h(x_c)=(u^{1(I)}(x_c))^2=0$ for the apex $x_c$ of $\mathcal{K}$. Once more, since $u^{1(I)}(x) \neq 0$, contradiction happens. And we arrive at the conclusion that $\omega_1=\omega_2$.

Similarly, if $F^{m(II)}_{u_ku_h}$ $(k\neq h)$ has a jump singularity, we can select $\tilde{\mathbf{u}}^{m(I)}$ which has non-trivial element only at the $h$ and $k$ positions. And we take $\tilde{\mathbf{u}}^{m(I)}=(1,1,0,\cdots,0)$ as an example. Now \eqref{eq:tp_Uniform_proofF2} becomes
\begin{equation}\notag
    \begin{cases}
		-d_1 \Delta u^{1(II)}_1-2F^{1,1(II)}_{1,u_1u_2}u^{1(I)}_1u^{1(I)}_2 = -d_1 \Delta u^{0(II)}_1-2F^{0(II)}_{1,u_1u_2}u^{1(I)}_1u^{1(I)}_2=0 &\quad \text{ in } \mathcal{K}_h,\\
        u_1^{1(II)}(x)=u_1^{0(II)}(x) &\quad \text{ on }\partial\mathcal{K}_h\backslash\partial B_h.
    \end{cases} 
\end{equation}
This formulation is equivalent to \eqref{eq:tp_AuxThmEq} with $w_n(x)=u^{n(II)}_1(x)(n=0,1) $, $P=0$, $q_n=F^{n(II)}_{1,u_2 u_1}$ and $h(x)=\big(u^{1(I)}_1 u^{1(I)}_2\big)(x)$. Since $F^{m(II)}_{1,u_2 u_1}\in\mathcal{A}$ and $u^{1(I)}_1(x), u^{1(I)}_2(x) \in C^{2+\alpha}(\mathcal{K}_h)$ under assumption, if 
\[F^{1,1(II)}_{1,u_2u_1}(x_c) \neq F^{0(II)}_{1,u_2u_1}(x_c), \]
there must be $h(x_c)=u^{1(I)}_2(x_c)u^{1(I)}_1(x_c)=0$ for the apex $x_c$ of $\mathcal{K}$ to satisfy the Theorem \ref{lem:tp_AuxThm}. 
However, we have chosen the initial function $f_{1,1}(x)$ and $f_{1,2}(x)$ in such a way that based on the maximum principle for the elliptic equation, neither $u^{1(I)}_1=0$ nor $u^{1(I)}_2=0$ holds on $\partial\omega_1\supset\partial\mathcal{K}$. This leads to a contradiction, implying that $\omega_1=\omega_2$. 
%%\ref{ForwardResult} refers to the well-posedness for original model.%%%%%%%%%%
Thus, we have once again established the uniqueness of $\omega$ uniquely.

Hence we can uniquely determine $\omega$ if any of the second-order coefficients of $\mathbf{F}$ has a jump singularity.
%if only one of the second-order coefficients $\mathbf{F}$ is unknown.
However, if all the second-order coefficient functions are identical at the point $x_c$, then $u^{1(II)}_i(x) = u^{0(II)}_i(x)$ on $\partial\omega_1$ for each $i$ and $\tilde{u}^{1(II)}_i\in C^{2+\alpha}(\Omega)$. In this case, the argument must be extended to higher-order variations. To analyze higher-order Taylor coefficients of $\mathbf{F}$, we consider the linearized systems of order $\ell\in\mathbb{N}$, with $l$-th order variations defined by
\begin{equation}\notag
    u_i^{(\ell)}(x):=\partial_{\varepsilon}^{\ell}u_i|_{\varepsilon=0},\  v_j^{(\ell)}(x):=\partial_{\varepsilon}^{\ell}v_j|_{\varepsilon=0} \quad\text{for }i=1,\dots,N; j=1,\dots,M.
\end{equation}
This recursive approach yields a sequence of systems, each of which can be applied again when a discontinuity is present in the corresponding higher-order Taylor coefficient $\mathbf{F}^{(\ell)}$.

We proceed with the case $\ell=3$. Under the assumptions above, the equality of measurement maps $\mathcal{M}_{\omega_1,\mathbf{F}^1}=\mathcal{M}_{\omega_2,\mathbf{F}^2}$ implies $u^{1,1(II)}_i=u^{0(II)}_i$ on $\partial\omega$.  Since the first- and second-order terms $u^{1(I)}_i, u^{1(II)}_i(i=1,\dots,N)$ are now fixed and smooth, the system governing the third-order variations $u^{n(III)}_i, u^{n(III)}_i (n=0,1)$ reduces to an elliptic equation whose remaining terms depend only on these lower-order solutions. Consequently, the unique continuation principle for elliptic equations applies once more, leading to a system of the form \eqref{eq:tp_AuxThmEq} with sufficiently regular coefficients.   A contradiction again arises if the third-order coefficients of $\mathbf{F}$ possess a jump singularity, thereby uniquely determining $\omega$.

The argument for orders $\ell (\ell \geq 4)$ follows identically. Hence, $\omega$ is uniquely determined whenever $\mathbf{F}$ has a jump singularity. Conversely, if $\mathbf{F}$ does not have a jump singularity, by definition, no such $\omega$ exists.

The proof is complete.

%%%%%%%%%%%%%%%%%%%%%%%%%%%%%%%%%%%%%%%%%%%%%%%%%%%%%%%%%%%%%%%%%%%%%%%%%%%%%%%%%%%%%
\subsection{Recovery of the coefficient functions}\label{sec:tp_nonsm_coe}
Having established the unique identifiability result for the intrinsic structure of anomalies in the given biological model, we proceed to consider the recovery for the coefficient function $\mathbf{F}$. In this subsection, we keep the notations in \eqref{eq:tp_tildeEXT2}, \eqref{eq:tp_hatEXT2} and \eqref{eq:tp_FExt}.

Repeating the argument above, from the first order linearisation \eqref{eq:tp_space_Linear1.2}, $u^{1(I)}_i$, $u^{2(I)}_i$ satisfy the following equations
\begin{equation}\notag
    \begin{cases}
		 -d_i \Delta u^{1(I)}_i(x)= -d_i \Delta u^{2(I)}_i(x)=0 & \quad\text{ in }  \mathcal{K}_h,\\
   \partial_{\nu} u^{1(I)}_{1}(x)=\partial_{\nu} u^{2(I)}_{1}(x) &\quad \text{ on }\partial\mathcal{K}_h\backslash\partial B_h,\\
         u^{1(I)}_{1}(x)=u^{2(I)}_{1}(x) &\quad \text{ on }\partial\mathcal{K}_h\backslash\partial B_h,\\
         u^{1(I)}_1(x)=u^{2(I)}_1(x) = f_{1,1}(x)>0 &\quad \text{ on }\partial\Omega.
    \end{cases}
\end{equation}
Then, it is clear that $u^{1(I)}_i = u^{2(I)}_i$ when the same boundary function $f_{1,1}$ is selected, and we employ $\bar{u}^{(I)}_1(x)$ for a unified representation. Given the positivity of the initial function $f_{1,i}$, we can infer the positivity of $\bar{u}^{(I)}_i(x)$ for each $i$.

To recover the second-order of the unknown coefficient function $\mathbf{F}$, we need to proceed to the system for $u^{1(II)}_i$ and $u^{2(II)}_i$, which satisfies:
\begin{equation}\label{eq:tp_sec5_u2}
    \begin{cases}
		-d_i \Delta u^{1(II)}_i-\sum\limits^{N}_{h =1} \sum\limits^{N}_{\substack{k =1, \\  k\neq h}} \big[
      (F^{1,1(II)}_{i,u_{k} u_h}+F^{1,1(II)}_{i,u_h u_{k}}) \bar{u}^{1(I)}_{h} \bar{u}^{1(I)}_{k}+  F^{1,1(II)}_{i,u_{h} u_{h}} (\bar{u}^{1(I)}_{h})^2   \big]=0 & \quad \text{ in } \mathcal{K}_h,\\
        -d_i \Delta u^{2(II)}_i-\sum\limits^{N}_{h =1} \sum\limits^{N}_{\substack{k =1, \\  k\neq h}} \big[
      (F^{2,1(II)}_{i,u_{k} u_h}+F^{2,1(II)}_{i,u_h u_{k}}) \bar{u}^{1(I)}_{h} \bar{u}^{1(I)}_{k}+  F^{2,1(II)}_{i,u_{h} u_{h}} (\bar{u}^{1(I)}_{h})^2  \big]=0 & \quad \text{ in } \mathcal{K}_h,\\
        u_i^{1(II)}(x)=u_i^{2(II)}(x),\  \partial_{\nu} u^{1(II)}_{i}(x)=\partial_{\nu} u^{2(II)}_{i}(x)  &\quad \text{ on }\partial\mathcal{K}_h\backslash\partial B_h.
    \end{cases}  	
\end{equation}
The cross-diffusion terms vanish as we chose the boundary function $\mathbf{g}$ as $\mathbf{0}$.

Once again, we identify the coefficients by selecting the appropriate $\bar{\mathbf{u}}^{1(I)}$. We start with $\bar{\mathbf{u}}^{1(I)}=(1,0,\cdots,0)$, which is non-trivial only in the first position. From there, we can shift this non-trivial element to recover the other coefficients  $F^{m(II)}_{u_iu_i}$ for $i=2,\cdots,N$.

Then, the equations \eqref{eq:tp_sec5_u2} becomes
\begin{equation}\label{eq:tp_case2_recoverF11}
    \begin{cases}
		-d_1 \Delta u^{1(II)}_1-F^{1,1(II)}_{1,u_1u_1}(\bar{u}^{1(I)}_1)^2 = -d_1 \Delta u^{2(II)}_1-F^{2,1(II)}_{1,u_1u_1}(\bar{u}^{1(I)}_1)^2=0 &\quad \text{ in } \mathcal{K}_h,\\
        u_1^{1(II)}(x)=u_1^{2(II)}(x) &\quad \text{ on }\partial\mathcal{K}_h\backslash\partial B_h.
    \end{cases} 
\end{equation}

Denote $w_n(x)=u^{n(II)}_1(x)(n=1,2) $, $P=0$, $q_n=F^{n,1(II)}_{1,u_1 u_1}$ and $h(x)=(\bar{u}^{1(I)}_1)^2(x)$, \eqref{eq:tp_case2_recoverF11} is simply \eqref{eq:tp_AuxThmEq}. If 
\[F^{1,1(II)}_{1,u_1u_1}(x_c) \neq F^{2,1(II)}_{1,u_1u_1}(x_c), \]
there must be $h(x_c)=(\bar{u}^{1(I)}(x_c))^2=0$ for the apex $x_c$ of $\mathcal{K}$, but this contradicts our choice of $f_{1,1}=1$ which gives $\bar{u}^{1(I)}_1=1$. Hence, it must be that $F^{1,1(II)}_{1,u_1u_1}(x_c) =F^{2,1(II)}_{1,u_1u_1}(x_c)$, then after extending to the part outside $\omega$, we arrive at $F^{1(II)}_{1,u_1u_1}(x_c) =F^{2(II)}_{1,u_1u_1}(x_c)$. 

Similarly, if $F^{n(II)}_{u_ku_h}$ $(k\neq h)$ is unknown, we can select $\tilde{\mathbf{u}}^{n(I)}$ which has non-trivial element only at the $h$ and $k$ positions. And we take $\tilde{\mathbf{u}}^{n(I)}=(1,1,0,\cdots,0)$ as an example. Now \eqref{eq:tp_Uniform_proofF2} becomes
\begin{equation}\notag
    \begin{cases}
		-d_1 \Delta u^{1(II)}_1-2F^{1,1(II)}_{1,u_1u_2}\bar{u}^{1(I)}_1\bar{u}^{1(I)}_2 = -d_1 \Delta u^{2(II)}_1-2F^{2,1(II)}_{1,u_1u_2}\bar{u}^{1(I)}_1\bar{u}^{1(I)}_2=0 &\quad \text{ in } \mathcal{K}_h,\\
        u_1^{1(II)}(x)=u_1^{2(II)}(x) &\quad \text{ on }\partial\mathcal{K}_h\backslash\partial B_h.
    \end{cases} 
\end{equation}
This formulation is equivalent to \eqref{eq:tp_AuxThmEq} with $w_n(x)=u^{n(II)}_1(x)(n=1,2) $, $P=0$, $q_n=F^{n,1(II)}_{1,u_2 u_1}$ and $h(x)=\big(\bar{u}^{1(I)}_1 \bar{u}^{1(I)}_2\big)(x)$. Since $F^{n(II)}_{1,u_2 u_1}\in\mathcal{A}$ and $\bar{u}^{1(I)}_1(x), \bar{u}^{1(I)}_2(x) \in C^{2+\alpha}(\mathcal{K}_h)$ under assumption, if 
\[F^{1,1(II)}_{1,u_2u_1}(x_c) \neq F^{2,1(II)}_{1,u_2u_1}(x_c), \]
there must be $h(x_c)=\bar{u}^{1(I)}_2(x_c)\bar{u}^{1(I)}_1(x_c)=0$ for the apex $x_c$ of $\mathcal{K}$ to satisfy the Theorem \ref{lem:tp_AuxThm}. 
However, we have chosen the initial function $f_{1,1}(x)$ and $f_{1,2}(x)$ in such a way that based on the maximum principle for the elliptic equation, neither $\bar{u}^{1(I)}_1=0$ nor $\bar{u}^{1(I)}_2=0$ holds on $\partial\omega_1\supset\partial\mathcal{K}$. This leads to a contradiction, implying that $F^{1,1(II)}_{1,u_2u_1}(x_c)=F^{2,1(II)}_{1,u_2u_1}(x_c)$. Extending them to the domain outside of $\omega$, we have $F^{1(II)}_{1,u_2u_1}(x_c)=F^{2(II)}_{1,u_2u_1}(x_c)$. 

In this way, by changing different $\bar{\mathbf{u}}^{(I)}$, we have uniquely determined $\mathbf{F}^{(II)}$.

Now, let $\bar{u}^{(II)}_i:=u^{1(II)}_i-u^{2(II)}_i$ for $i=1,\dots,N$. We can obtain the following equations:
\begin{equation}\notag
   \begin{cases}
       -d_i \bar{u}^{(II)}_i(x)=0 & \quad \text{ in } \mathcal{K}_h,\\
        \bar{u}^{(II)}_i(x)=0 &\quad \text{ on }\partial\mathcal{K}_h\backslash\partial B_h,
   \end{cases} 
\end{equation}
which implies us that $u^{1(II)}_i=u^{2(II)}_i$, and we denote it by $\bar{u}^{(II)}_i$.

%In order to recover the higher-order variation coefficient function of $\mathbf{F}$, we keep the linearisation by $\ell=3$. It is important to note that under the aforementioned assumptions, the system for the third-order variation of $u^{n(III)}_1 (n=1,2)$ is simply an elliptic equation with the other terms involving only the lower term of $\bar{u}^{(I)}_i, \bar{u}^{(II)}_i$. Therefore, we can once again apply the unique continuation principle for elliptic equations and obtain a system in the form of \eqref{eq:tp_AuxThmEq} with sufficiently regular coefficients. Due to our choice of positive $f_{1,i}$, the condition $h(x_c)=\sum^{N}_{i=1}(u^{(I)}_i)^{l_i}(x_c)=0$ (where $l_i=0,1,2,3$ and $\sum^{N}_{i=1}l_i=3$) leads to a contradiction. Thus we can uniquely determine the third-order variation of the coefficient function $\mathbf{F}$. 

Similarly, we can recover the higher-order $ (\ell \geq 3)$ variation of $\mathbf{F}$ by repeating the argument. %We start by establishing the equality  of $u^{1(\ell')}_i=u^{2(\ell')}_i$ for $\ell'<\ell$ on $\partial\omega$. This leads to an elliptic equation in the $\ell$-th order variation system for $u^{n(\ell)}_i$ $(n=1,2)$, with the other terms involving only the lower order $\{u^{m(\ell')}_i\}_{\ell'}$. By applying the unique continuation principle, we obtain an equation in the form of \eqref{eq:tp_AuxThmEq} with sufficiently regular coefficients. The application of Theorem \ref{lem:tp_AuxThm} yields $\sum^{N}_{i=1}(u^{(I)}_i)^{\ell_i}(x_c)=0$ (where $\ell_i =0,1,\dots,N$ and $\sum^{N}_{i=1}\ell_i=\ell$). However, this contradicts our choice for each $f_{1,i}$. Hence, we have uniquely determined $\mathbf{F}^{\ell}$.

The proof is complete. $ \hfill{\square}$

\noindent\textbf{Acknowledgment.} 
The work of H. Liu is supported by NSFC/RGC Joint Research Scheme, N CityU101/21, ANR/RGC Joint Research Scheme, A-CityU203/19, and the Hong Kong RGC General Research Funds (projects 11311122, 11303125 and 11300821). The work of C. W. K. Lo is supported by the National Natural Science Foundation of China (No. 12501660) and the Guangdong Province ``Pearl River" Talent Recruitment Program (Top youth talent) (No. 2024QN11X268).

\bibliographystyle{plain}
\bibliography{reference}

\end{document}